\documentclass{amsart}
\usepackage{amssymb,latexsym,amsmath}
\usepackage{amsthm}
\usepackage{xcolor}
\usepackage{enumitem}
\usepackage[all]{xy}

\def\ZZ{{\mathbb Z}}

\newtheorem{formula}{}[section]
\newtheorem{proposition}[formula]{Proposition}
\newtheorem{definition}[formula]{\indent Definition}
\newtheorem{corollary}[formula]{\indent Corollary}
\newtheorem{remark}[formula]{\indent Remark}
\newtheorem{lemma}[formula]{\indent Lemma}
\newtheorem{theorem}[formula]{\indent Theorem}

\newtheorem{example}[formula]{Example}
\newtheorem{convention}[formula]{Convention}

\def\thrm{\begin{theorem}}
\def\thrml#1{\begin{theorem}\label{#1}}
\def\ethrm{\end{theorem}}
\def\rmrk{\begin{remark}}
\def\rmrkl#1{\begin{remark}\label{#1}}
\def\ermrk{\end{remark}}
\def\dfntn{\begin{definition}}
\def\dfntnl#1{\begin{definition}\label{#1}}
\def\edfntn{\end{definition}}
\def\nmrt{\begin{enumerate}}
\def\enmrt{\end{enumerate}}

\def\qtn{\begin{equation}}
\def\qtnl#1{\begin{equation}\label{#1}}
\def\eqtn{\end{equation}}
\def\lmm{\begin{lemma}}
\def\lmml#1{\begin{lemma}\label{#1}}
\def\elmm{\end{lemma}}
\def\crllr{\begin{corollary}}
\def\crllrl#1{\begin{corollary}\label{#1}}
\def\ecrllr{\end{corollary}}

\newcommand{\Sol}{\operatorname{Sol}}

\title{Minimal solutions of tropical linear differential systems}
\author{Dima Grigoriev \and Cristhian Garay L\'opez}
\date{ \today}

\begin{document}


\begin{abstract}
We introduce and study minimal (with respect to inclusion) solutions of systems of tropical linear differential equations. We describe the set of all minimal solutions for a single equation. It is shown that any tropical linear differential equation in a single unknown has either a solution or a solution at infinity. For a generic system of $n$ tropical linear differential equations in $n$ unknowns, upper and lower bounds on the number of minimal solutions are established. The upper bound involves inversions of a family of permutations which generalize inversions of a single permutation.  For $n=1, 2$, we show that the bounds are sharp.  
\end{abstract}
\maketitle
{\bf keywords}: tropical linear differential equations, minimal solutions, inversion of a family of permutations \vspace{1mm}


\section{Introduction}
We recall \cite{G17} that a tropical ordinary linear differential equation (abbr. TLDE) of order $k\in \ZZ_{> 0}$ is an expression of the form

\begin{eqnarray}\label{1}
P:=\min_{0\leq i\leq k} \{a_i+u^{(i)}\}
\end{eqnarray}
where $a_i\in \ZZ$. For a set $S\subset \ZZ_{\ge 0}$ define a valuation 
$$val_S(i):= \min \{s-i\\ |\ i\le s\in S\},$$ 
provided that  such min exists, and $val_S(i)=\infty$ otherwise.
\noindent We say that $S$ is a {\it tropical solution} of (\ref{1}) if the minimum in $\min_{0\le i\le k} \{a_i+ val_S(i)\}$ is attained at least twice.

If a classical ordinary linear differential equation (abbr. LDE)
\begin{equation}
\label{CLDE}
 \sum_{0\le i\le k} {\mathcal A}_i \frac{d^i w}{dt^i}=0   
\end{equation}
\noindent   with Laurent series coefficients ${\mathcal A}_i=t^{a_i} \sum_{j\ge 0} {\mathcal A}_{ij} t^j,\ {\mathcal A}_{i0}\neq 0$ has a power series $w$ as its solution, then the support $S=supp(w)\subset \ZZ_{\ge 0}$ is a tropical solution of (\ref{1}). In the sequel, talking about tropical  solutions  of (\ref{1}), we omit the adjective ``tropical'' for brevity.

We call a solution $S$ of (\ref{1}) {\it minimal} when $S$ is minimal among solutions with respect to inclusion. We say that the equation (\ref{1}) is {\it holonomic} if it has just a finite number of minimal solutions. 

Note that a system of (partial) LDE is called holonomic if its space of solutions has a finite dimension. In particular, an ordinary LDE is always holonomic, which is not necessary the case for its tropical counterpart, as we show  in Example~\ref{non-holonomic}. 

  Similarly, a TLDE in $n\geq1$ unknowns is an expression 

\begin{equation}
    \label{22_P}
    P=\bigoplus_{\substack{1\leq j\leq n\\
0\leq i\leq k_j}}a_{i,j}\odot u_{j}^{(i)}=\min_{\substack{1\leq j\leq n\\
0\leq i\leq k_j}}\{{a_{i,j}}+u_{j}^{(i)}\}, \text{ where }a_{i,j}\in \mathbb{Z}.
\end{equation}

A tuple $(S_1,\ldots,S_n)\in \mathcal{P}(\mathbb{Z}_{\geq0})^n$ is a solution of \eqref{22_P} if and only if the value $\min_{\substack{1\leq j\leq n\\
0\leq i\leq k_j}}\{{a_{i,j}}+\text{val}_{S_j}(i)\}$ is attained at least twice. We call a solution $(S_1,\ldots,S_n)$ of (\ref{22_P}) {\it minimal} when $S_j$ is minimal among solutions with respect to inclusion for all $j\in[n]$. 

 If a classical ordinary linear differential equation
\begin{equation}
\label{CLDE_n>1}
 \sum_{\substack{1\leq j\leq n\\
0\leq i\leq k_j}} {\mathcal A}_{i,j} \frac{d^i w_j}{dt^i}=0   
\end{equation}
\noindent   with coefficients $t^{a_{ij}} \sum_{k\ge 0} {\mathcal A}_{ijk} t^k={\mathcal A}_{i,j}\in  {K}[\![t]\!]$ (where ${\mathcal A}_{ij0}\neq 0$, $K$ a field  of characteristic zero) has a power series $w=(w_1,\ldots,w_n)$ as its solution, then the support $(supp(w_1),\ldots,supp(w_n))$ is a  solution of (\ref{22_P}).

Consider now the ring of formal power series in $s\geq1$ variables ${K}[\![t_1,\ldots,t_s]\!]$. Let $Sol(\Sigma)\subset {K}[\![t_1,\ldots,t_s]\!]^n$ be the set of classical formal power series solutions  of an homogeneous linear system $\Sigma$ of LDE in $n\geq1$ unknowns $u_1,\ldots,u_n$ with coefficients in ${K}[\![t_1,\ldots,t_s]\!]$.

In \cite{ABFG}, the authors  show that under certain circumstances ($\Sigma$ has to be of differential type zero and the cardinality of $K$ has to be large enough) then the {\it tropicalization} $trop(Sol(\Sigma))\subset\mathcal{P}(\mathbb{Z}_{\geq0}^s)^n$ of $Sol(\Sigma)$ (i.e. its  whole family of supports) is a $\mathbb{B}$-semimodule that has the structure of an (infinite) matroid.

The result from \cite{ABFG} can be presented as the next step in developing the theory of tropical differential algebraic geometry, right after the so-called fundamental theorems of tropical differential algebraic geometry (both the ordinary \cite{Fundamental_theorem_TropDiffAlgGeom} and the partial version \cite{falkensteiner2020fundamental}), which  offer a connection between this  {\it tropicalization} $trop(Sol(\Sigma))$, and the set of  solutions of certain systems of TLDEs of the form \eqref{22_P}.  

Also in \cite{ABFG}, the authors say that from a combinatorial perspective, it makes sense to study the structure of $Sol(\Sigma)\subset\mathcal{P}(\mathbb{Z}_{\geq0}^s)^n$ associated to tropical linear systems $\Sigma$ consisting of TLDEs of the form \eqref{22_P}.  

In this paper we study this problem for the case $s=1$.  If $\Sigma$ is such a system, then the set $\mu(Sol(\Sigma))$ of minimal  solutions of a system $\Sigma$ {\it generates} (from an order-theoretic perspective) the whole set $Sol(\Sigma)$ of  solutions of $\Sigma$, see Remark \ref{rem:generation}. We show that if $n>1$, then the set $\mu(Sol(P))$ of tropical solutions of a single TLDE   contains in a natural way certain points of a  matroid, not infinite, but valuated in this case (Corollary~\ref{eq:circuits_n>1}).

\subsection{Results}
We fix $\mathbf{k}=(k_1,\ldots,k_n)\in \mathbb{Z}_{>0}^n$ and we consider $P=P_1(u_1)\oplus\cdots \oplus P_n(u_n)$ a TLDE
as in \eqref{22_P}, such that the differential order of $P_j(u_j)$ is $k_j\in\mathbb{Z}_{>0}$ for $j\in[n]$.

The first tool in our analysis is the following decomposition of the set $\mu(\Sol(P))$ of non-zero minimal (with respect to inclusion) solutions of $P$ introduced in Definition \ref{eq:decompo} 
\begin{equation}
\label{eq:intro_desc}
    \mu(\Sol(P))=F(P)\sqcup\mathcal{C}_{\mathbb{Z}_{\geq0}}(P),
\end{equation} where $F(P)$ is a  finite set, and $\mathcal{C}_{\mathbb{Z}_{\geq0}}(P)$ is the set of minimal solutions of $P$ having order at least $\mathbf{k}$.


The second tool in our analysis is a family of tropical linear (homogeneous) polynomials $\{A_{\alpha,j},\:j\in[n],\:0\leq \alpha\}$ introduced in Definition \ref{def:linpols_n>1}:
\begin{equation}
\label{eq:intro_polynomials}
    A_{\alpha,j}=\bigoplus_{i=0}^{k_j}val_{t_j^\alpha}(i)\odot x_{i,j}\in \mathbb{T}[x_{i,j}\::\:1\leq j\leq n,\:0\leq i\leq k_j, \alpha].
\end{equation}

Our main results for the first part are complete descriptions of $\mu(\Sol(P))$. If $n>1$, the expression $P(u_1,\ldots,u_n)=P_1(u_1)\oplus\cdots \oplus P_n(u_n)$ tells us that $\bigcup_j\mu(\Sol(P_j))\subset \mu(\Sol(P))$, so we see that the case of $n>1$  depends on the case $n=1$. \vspace{1mm}


{\bf Theorem (Theorem \ref{thm:s_n=1} in the text). } 
Consider $P=P(u)$ as in (\ref{1}). Denote by $V(A_{\alpha})\subset \ZZ^{k+1}$ the tropical hyperplane defined by  the tropical polynomial $A_{\alpha}:=\min_{0\le i\le k, \alpha} \{a_i+\alpha-i\},\ 0\le \alpha$, cf. (\ref{eq:intro_polynomials}).   Then 
\begin{enumerate}[label=\roman*)] 
    \item The set $\mathcal{C}_{\mathbb{Z}_{\geq0}}(P)$ satisfies \begin{equation*}
    \mathcal{C}_{\mathbb{Z}_{\geq0}}(P)=\begin{cases}
        \emptyset, &\text{ if }P\notin V(A_k),\\
        \mathbb{Z}_{\geq0}\odot\{k\}=\{\{k+i\}\::\:i\in \mathbb{Z}_{\geq0}\},&\text{ otherwise}.\\
    \end{cases}
\end{equation*} 

    \item The set $F(P)$ is finite, and it consists of polynomials $S=t^p+t^q$, with $0\leq p\leq q$ satisfying
    \begin{enumerate}
        \item $0\leq p\leq q< k$, or 
        \item $0\leq p< k\leq q$ for $q=q(p)$ unique (only when $P\notin V(A_k)$).
    \end{enumerate}
\end{enumerate} \vspace{1mm}

{\bf Theorem (Theorem \ref{thm:s_n>1}  in the text). } Consider $P=\bigoplus_{j=1}^nP_j(u_j)$ for $n>1$.
 The set $\mathcal{C}_{\mathbb{Z}_{\geq0}}(P)$  satisfies \begin{equation*}
        \mathcal{C}_{\mathbb{Z}_{\geq0}}(P)=\bigcup_{P\in V(A_{k_j,j})}\mathbb{Z}_{\geq0}\odot t_j^{k_j}\cup\bigcup_{A_{k_b,b}(P)\leq A_{k_a,a}(P)}\mathbb{Z}_{\geq0}\odot(t_a^{k_a}+t_b^{k_b+A_{k_a,a}(P)-A_{k_b,b}(P)}).
    \end{equation*}
    The set $F(P)$ is finite. Furthermore, if $S\in F(P)$, then either 
\begin{enumerate}
\item $S\in \bigcup_jF(P_j)$, or
\item  $S=t_i^p+t_j^q$ satisfies
\begin{enumerate}
    \item $(p,q)<(k_i,k_j)$,
    \item if $p<k_i$ and $q=q(p)\geq k_j$ is unique (or  $q< k_j$ and $p=p(q)\geq k_i$ is unique).
\end{enumerate}
\end{enumerate}

In particular, $\mathcal{C}_{\mathbb{Z}_{\geq0}}(P)$ is contained in a valuated matroid $\mathcal{C}_{\mathbb{R}}(P)\subseteq \mathbb{T}^n$.

We also use the decomposition \eqref{eq:intro_desc} and the polynomials \eqref{eq:intro_polynomials} to give a definition of regularity for the case $n=1$ in Definition \ref{def_reg_n=1}, and  for  $n>1$  in  Definition \ref{def_reg_n>1}. 

Afterwards we consider finite systems $\Sigma$  of $n\geq1$ TLDEs in $n$ unknowns of the form 

\begin{equation}\label{4}
P_l=\bigoplus_{1\le j\le n,\ 0\le i\le k_{j}}a_{ijl}\odot u_j^{(i)}=\min_{1\le j\le n,\ 0\le i\le k_{j}} \{a_{ijl}+u_j^{(i)}\},\ 1\le l\le n.
\end{equation}

We say that $\Sigma$ is holonomic if it has a finite number of minimal solutions. We deduce (Proposition~\ref{p:holonomic}) that the system $\Sigma=\{P\}$ is holonomic if and only if $P\notin V(A_k)$. 

 Our main results for the second part include upper and lower bounds on the number of minimal solutions for a generic  system \eqref{4}. The upper bound involves inversions of a family of permutations which generalize inversions of a single permutation. We show that the bounds are sharp for $n=1, 2$. For $n=1$ the maximal number of minimal solutions of a generic regular TLDE (\ref{1}) equals $k$ (Theorem~\ref{complete}), and for $n=2$, we give the following result.

{\bf Corollary} (Corollary~\ref{gap}  in the text). The maximal number of minimal solutions of a generic regular system (\ref{4}) for $n=2$ equals $\frac{(k_1+k_2)(k_1+k_2+1)}{2}$.

For generic regular systems of $n>2$ TLDEs of the type (\ref{4}) we provide lower and upper bounds on the  number of minimal solutions (unlike the case $n=2$ there is a gap between obtained lower and upper bounds, and 'it would be interesting to diminish this gap).

\subsection{Roadmap}
Section 2  is about preliminaries. After that the paper can be divided in two parts  as follows. The first part consists of section  \ref{sect:msstlde}, and it describes the structure of the set $\mu(\Sol(P))$ as well as the definition of regularity for TLDEs.

The second part consists of sections 4-6. In section~\ref{four} we start our study of systems of TLDEs \eqref{4} by introducing our definition of generic regular system (cf. Definition~\ref{regular}); after that, we study carefully the bounds for the case $n=2$.

The remaining sections are about providing bounds for generic regular systems of $n>2$ TLDE of the type (\ref{4}): in section~\ref{five} we establish a lower bound on the number of minimal solutions (Theorem~\ref{lower}), and in section~\ref{six} we prove an upper bound (Theorem~\ref{upper_improved}). The proof relates the upper bound with a new concept of inversions of a family of permutations (see Definition~\ref{def:inversion}, Theorem~\ref{inversion}) generalizing the customary inversions of a permutation.

\subsection{Conventions}
If  $E$ is a non-empty set, we denote by $|E|$ its cardinality and by  $\mathcal{P}(E)$ its power set. If $E=\{1,\ldots, n\}$ is finite, we denote it by $[n]$. We also use the following notation 

\begin{equation*}
    \binom{[n]}{i}=\begin{cases}
        \{X\subset [n]\::\:|X|=i\},&\text{ if }0\leq i\leq n,\\
        \emptyset,&\text{ otherwise}.\\
    \end{cases}
\end{equation*}

If $A\subset B$ are sets, we denote by $A\setminus B$ its relative complement. We denote by $Sym(n)$ the symmetric group, and we view a permutation from $Sym(n)$ as a bijection of the set $[n]$ with itself.



\section{Preliminaries}
\subsection{$\mathbb{B}$-semimodules and $\mathbb{B}$-algebras}
\label{Sect_PB}
 In this section we  will introduce the main algebraic objects of this work, which are   (semi)modules and algebras over the Boolean semifield $\mathbb{B}=(\{0<\infty\},\odot=+,\oplus=\min)$. This language provides an unified approach that mixes both order-theory and commutative algebra.

    \begin{definition}
    A $\mathbb{B}$-{\bf semimodule} is a triple $M=(M,+,0)$ consisting of a commutative, idempotent semigroup $(M,+)$ with identity element $0$. A $\mathbb{B}$-subsemimodule of $M$ is a subset $N\subseteq M$ which is itself a $\mathbb{B}$-semimodule. A morphism of $\mathbb{B}$-semimodules is a map $f:M_1\xrightarrow[]{}M_2$ that satisfies $f(0)=0$ and $f(a+b)=f(a)+f(b)$.
\end{definition}


Any $\mathbb{B}$-semimodule $M$ bears a canonical max order : $a\leq_M b$ if and only if $a+b=b$; note that $0\in M$ is always the unique bottom under this order. Sometimes we  use the opposite order $a\leq_M b$ if and only if $a+b=a$, and in this case we will denote the identity element by $\infty$.

 Thus we obtain a poset $(M,\leq)$ (we will drop the $M$ from $\leq_M $). In particular, if  $f:M_1\xrightarrow[]{}M_2$  is a morphism of $\mathbb{B}$-semimodules, then $f$ preserves  (respectively reverses) these orders if they agree (respectively if they disagree).

\begin{example}
    The power set $\mathcal{P}(\mathbb{Z}_{\geq0})$ of $\mathbb{Z}_{\geq0}$ equipped with union as binary operation is a $\mathbb{B}$-semimodule (the identity is $\emptyset$). Note that for $S,T\in \mathcal{P}(\mathbb{Z}_{\geq0})$, we have $S\leq T$ if and only if $S\subseteq T$.
\end{example}

For $n\geq1$ fixed, we will denote by $\mathcal{P}(\mathbb{Z}_{\geq0})^n$ the $\mathbb{B}$-semimodule which is the $n$-fold product of $\mathcal{P}(\mathbb{Z}_{\geq0})$. The identity is $(\emptyset,\ldots,\emptyset)$, and given $S,T\in \mathcal{P}(\mathbb{Z}_{\geq0})^n$ with $S=(S_1,\ldots,S_n)$ and $T=(T_1,\ldots,T_n)$, then  $S\leq T$ if and only if $S_i\subseteq T_i$ for $i\in[n]$.

\begin{definition}
    Let  $M$ be a $\mathbb{B}$-subsemimodule of $\mathcal{P}(\mathbb{Z}_{\geq0})^n$.  We denote by $\mu(M)=\min_{\leq}\{M\setminus \{(\emptyset,\ldots,\emptyset)\}\}$ the set of non-zero minimal elements of $(M,\leq)$.
\end{definition}

Given $\emptyset\neq X\subset M$, we denote by $\langle X\rangle^{\uparrow}=\{n\in M\::\:x\leq n\text{ for some }x\in X\}$. Note that $\langle X\rangle^{\uparrow}=M$ if and only if $(\emptyset,\ldots,\emptyset)\in X$. 

\begin{proposition}
\label{eq:sum}
    Let  $M\neq \{(\emptyset,\ldots,\emptyset)\}$ be a $\mathbb{B}$-subsemimodule of $\mathcal{P}(\mathbb{Z}_{\geq0})^n$. Then \begin{equation*}
         M=\langle  \mu(M)\rangle^{\uparrow}\cup\{(\emptyset,\ldots,\emptyset)\}
     \end{equation*}
\end{proposition}
\begin{proof}
We just need to prove the inclusion $M\setminus \{(\emptyset,\ldots,\emptyset)\}\subset \langle  \mu(M)\rangle^{\uparrow}$, which follows from the fact that there are no infinite descending chains in the poset $(\mathcal{P}(\mathbb{Z}_{\geq0})^n,\leq)$.
\end{proof}

A $\mathbb{B}$-algebra is equivalent to an (additively) idempotent semiring. The canonical order of a   $\mathbb{B}$-algebra  is   the same as its order as  a $\mathbb{B}$-semimodule.

\begin{example}
    The $\mathbb{B}$-semimodule $\mathcal{P}(\mathbb{Z}_{\geq0})$ becomes a $\mathbb{B}$-algebra when we endow it with the product $ST=\{i+j\::\:i\in S,\:j\in T\}$ (this denotes the Minkowski sum of subsets of $\mathbb{Z}_{\geq0}$). The multiplicative  identity is $\{0\}$. 
    
    For $n\geq1$ fixed, we will denote by $\mathcal{P}(\mathbb{Z}_{\geq0})^n$ the $\mathbb{B}$-algebra which is the $n$-fold product of the $\mathbb{B}$-algebra $\mathcal{P}(\mathbb{Z}_{\geq0})$.
\end{example}

We denote by $\mathbb{T}$ the $\mathbb{B}$-algebra  $(\mathbb{R}\cup\{\infty\},\odot=+,\oplus=\min)$ known as the tropical semifield.  Note that $\mathbb{T}$ bears the min order: $a\leq b$ if and only if $a\oplus b=a$. We also denote by $\mathbb{T}^*$ the tropical torus $\mathbb{T}\setminus\{\infty\}=(\mathbb{R},0,\odot=+)$.

Recall that an expression $a=a_1\oplus\cdots\oplus a_n=min\{a_1,\ldots,a_n\}$  in $\mathbb{T}$ vanishes (tropically), if either 
\begin{enumerate}[label=\roman*)]
    \item $a=\infty$, or
    \item $a\neq\infty$, and $a=a_i=a_j$ for some $i,j\in[s]$ with $i\neq j$,
\end{enumerate}


Condition ii) above is equivalent to say that $n\geq2$, and the minimum in $a=min\{a_1,\ldots,a_n\}$ is achieved at least twice.

If we use set-theoretic notation $a=min\{a_1,\ldots,a_n\}$ instead of $a=a_1\oplus\cdots\oplus a_n$, we assume that the factors may not necessarily be distinct

 If $f\in\mathbb{T}[x_1,\ldots,x_n]$ is a tropical polynomial, say $f=\bigoplus_\alpha a_\alpha \odot x^{\odot \alpha}$ it induces an evaluation map $ev_f:\mathbb{T}^n\xrightarrow[]{}\mathbb{T}$ sending $p=(p_1,\ldots,p_n)\in\mathbb{T}^n$ to $ev_f(p)=\bigoplus_\alpha a_\alpha \odot p^{\odot \alpha}$. We denote by $V(f)\subset \mathbb{T}^n$ its corner locus, equivalently, the set of points $(p_1,\ldots,p_n)\in\mathbb{T}^n$  such that $ev_f(p)$ vanishes in $\mathbb{T}$.   

 We denote by $\odot:\mathbb{T}^*\times \mathbb{T}^n\xrightarrow[]{}\mathbb{T}^n$ the action of the tropical torus $\mathbb{T}^*$ on $\mathbb{T}^n$ given by $\lambda\odot p=(p_1+\lambda,\ldots,p_n+\lambda)$. Note that if a polynomial $f$ is homogeneous, then $\mathbb{T}^*$ acts naturally on $V(f)$: if $p\in V(f)$ and $\lambda\in \mathbb{T}^*$, then $\lambda\odot p=(p_1+\lambda,\ldots,p_n+\lambda)\in V(f)$.

\begin{definition}
\label{def:notvanishstr}
    We say that $a=a_1\oplus\cdots\oplus a_n=min\{a_1,\ldots,a_n\}$ {\bf vanishes weakly} if $a_i=a_j\neq\infty$ for some $i\neq j$.
\end{definition}

\begin{remark}
\label{rem:wv_tropical}
    Note that if $a=a_1\oplus\cdots\oplus a_n$ vanishes, then it vanishes weakly, but the converse is not true (tropical vanishing is weak vanishing plus the condition $a=a_i=a_j$). An equivalent way to say  that $a=a_1\oplus\cdots\oplus a_n$ does not vanish weakly is if all the factors $a_i\neq\infty$ are distinct, so the locus of all $(a_1,\ldots,a_n)\in(\mathbb{T^*})^n$  such that $a=a_1\oplus\cdots\oplus a_n$ weakly vanishes is the tropical hypersurface $V(\bigodot_{1\leq i<j\leq n}x_i\oplus x_j)$. 
\end{remark}


\subsection{Differential algebra over $\mathbb{B}$}

Fix $n\geq1$. A TLDE in $n$ unknowns $\{u_1,\ldots,u_n\}$ (differential variables) and of differential order $\mathbf{k}=(k_1,\ldots,k_n)\in(\mathbb{Z}_{>0})^n$, is a formal  expression $P$ as in \eqref{22_P}. This induces a differential evaluation map 
\begin{equation*}
    \begin{aligned}
        \text{dev}_P:\mathcal{P}(\mathbb{Z}_{\geq0})^n&\xrightarrow[]{}\mathcal{P}(\mathbb{Z}_{\geq0})\\
        (S_1,\ldots,S_n)&\mapsto P(u_{j}^{(i)}\mapsto d^i(S_j)),
    \end{aligned}
\end{equation*}

where  $d:\mathcal{P}(\mathbb{Z}_{\geq0})\xrightarrow[]{}\mathcal{P}(\mathbb{Z}_{\geq0})$ is the operator $d(S)=\{i-1\::\:i\in S,\:i-1\geq0\}$. Note that  $\text{dev}_P$ is a homomorphism of $\mathbb{B}$-semimodules.  

We denote by $ord:\mathcal{P}(\mathbb{Z}_{\geq0})\xrightarrow[]{}\mathbb{T}$ the  map defined by $ord(\emptyset)=\infty$ and $ord(S)=min(S)$ otherwise. Note that $ord$ is an (order-reversing) homomorphism of $\mathbb{B}$-semimodules, and its image is contained in $\mathbb{Z}\cup\{\infty\}$. We define $\text{trop}_P:\mathcal{P}(\mathbb{Z}_{\geq0})^n\xrightarrow[]{}\mathbb{T}$ by $\text{trop}_P(S)=ord\circ\text{dev}_P(S)$, which is an (order-reversing) homomorphism of $\mathbb{B}$-semimodules. If we denote  $ord(d^i(S_j))=\text{val}_{S_j}(i)$, then 
\begin{equation}
\label{eq:trop_def}
    trop_P(S_1,\ldots,S_n):=\bigoplus_{\substack{1\leq j\leq n\\
0\leq i\leq k_j}}a_{i,j}\odot \text{val}_{S_j}(i)=\min_{\substack{1\leq j\leq n\\
0\leq i\leq k_j}}\{{a_{i,j}}+\text{val}_{S_j}(i)\}.
\end{equation}

A tuple $S=(S_1,\ldots,S_n)\in \mathcal{P}(\mathbb{Z}_{\geq0})^n$ is a solution of \eqref{22_P} if and only if the value $trop_P(S)=trop_P(S_1,\ldots,S_n)$ from \eqref{eq:trop_def} vanishes tropically (see Section \ref{Sect_PB}). We denote by $\Sol(P)\subset \mathcal{P}(\mathbb{Z}_{\geq0})^n$ the set of solutions of $P$.

Now fix $m\geq1$, and consider a system $\Sigma=\{P_1,\ldots,P_m\}$ of TLDEs in $n\geq1$ unknowns. We denote by $\Sol(\Sigma)=\bigcap_l\Sol(P_l)$ the set of solutions of $\Sigma$.

\begin{proposition}
    The set $\Sol(\Sigma)$ is a $\mathbb{B}$-subsemimodule of  $\mathcal{P}(\mathbb{Z}_{\geq0})^n$.
\end{proposition}
\begin{proof}
    We have that if $P$ is linear, then $(\emptyset,\ldots,\emptyset)\in\Sol(P)$, since $trop_P(\emptyset,\ldots,\emptyset)=\infty$. If $S,T\in \Sol(P)$, then $trop_P(S\cup T)=trop_P(S)\oplus trop_P(T)$ since $trop_P$ is a homomorphism. The result follows from the fact that any intersection of $\mathbb{B}$-subsemimodules of $\mathcal{P}(\mathbb{Z}_{\geq0})^n$ is again a $\mathbb{B}$-subsemimodule of it.
\end{proof}

\begin{remark}
\label{rem:generation}
    It follows from Proposition \ref{eq:sum} that the $\mathbb{B}$-semimodule $\Sol(\Sigma)$ is generated (from an order-theoretic perspective) by the set of minimal elements $\mu(\Sol(\Sigma))=min_{\leq}(\Sol(\Sigma)\setminus\{(\emptyset,\ldots,\emptyset)\})$ of the poset $(\Sol(\Sigma),\leq)$. This is: $\Sol(\Sigma)=\langle  \mu(\Sol(\Sigma))\rangle^{\uparrow}\cup\{(\emptyset,\ldots,\emptyset)\}$.
\end{remark}

\begin{definition}
    We call a system $\Sigma$ {\bf holonomic} if it has a finite number of minimal solutions.
\end{definition}

In the coming  sections we will compute $\mu(\Sol(\Sigma))$ for different cases of systems $\Sigma$ of $m\geq1$ polynomials in $n\geq1$ variables, and we will bounds for those which are holonomic.

We will now introduce some functions on the space of parameters of the TLDEs $P$ as in \eqref{22_P} which will be used in the following. We shall start with the case $n=1$.
\begin{definition}
\label{def:linpols_n=1}
   For $k\in\mathbb{Z}_{>0}$ fixed, we consider the following family of tropical linear polynomials $\{A_j\::\:j=0,\ldots,k\}$, defined by
\begin{equation*}
    A_j:=\bigoplus_{i=0}^kval_{\{j\}}(i)\odot x_i\in \mathbb{T}[x_0,\ldots,x_k].
\end{equation*} 
\end{definition}

The space of parameters of TLDEs $P(u)=\bigoplus_{i=0}^{k}a_{i}\odot u^{(i)}$ of order $k$ can be identified with $\mathbb{Z}^{k+1}$ sitting inside  $\mathbb{T}^{k+1}$. If we denote by $c(P)=(a_{i})_{0\leq i\leq k}\in \mathbb{Z}^{k_j+1}$ the vector of coefficients of $P$, then we can evaluate the polynomial $A_{j}$ at the vector $c(P)$, and we have 
\begin{equation}
\label{eq_pol_duality}
    A_{j}(c(P))=\bigoplus_{i=0}^{k}val_{\{j\}}(i)\odot a_{i}=\min_{0\leq i\leq min\{j,k\}}\{a_{i}+(j-i)\}=trop_{P}(\{j\})
\end{equation}

We now discuss the case $n>1$: we assume that  $P$ is as in \eqref{22_P} and has  fixed differential order $\mathbf{k}=(k_1,\ldots,k_n)\in(\mathbb{Z}_{>0})^n$. We shall use \eqref{eq_pol_duality} to define the family $\{A_{\alpha,j}\::\:j\in[n],\:0\leq \alpha\leq k_j\}$. For $j\in[n]$ and $\alpha\in \mathbb{Z}_{\geq0}$ we consider $S_{\alpha,j}=(S_1,\ldots,S_n)\in\mathcal{P}(\mathbb{Z}_{\geq0})^n$ defined by $S_l=\emptyset$ if $l\neq j$, and $S_j=\{\alpha\}$. So $A_{\alpha,j}$ must satisfy $A_{\alpha,j}(c(P)):=trop_{P}(S_{\alpha,j})$, where $c(P)\in \prod_{j=1}^n\mathbb{Z}^{k_j+1}$ is the vector of coefficients of $P$,  which is a point of the space of parameters of these TLDEs. 

\begin{definition}
    \label{def:linpols_n>1}
    We consider the family of tropical linear polynomials $\{A_{\alpha,j}\::\:j\in[n],\:0\leq \alpha\leq k_j\}$  defined by
\begin{equation*}
    A_{\alpha,j}=\bigoplus_{i=0}^{k_j}val_{S_{\alpha,j}}(i)\odot x_{i,j}\in \mathbb{T}[x_{i,j}\::\:1\leq j\leq n,\:0\leq i\leq k_j].
\end{equation*}
\end{definition}

\begin{convention}
\label{conv:rep_sols}
Representing elements $S=(S_1,\ldots,S_n)\in\mathcal{P}(\mathbb{Z}_{\geq0})^n$ for $n>1$ sometimes becomes cumbersome.  We propose using multiplicative notation $S=S_1(t_1)+\cdots+ S_n(t_n)$, where $S_j(t_j)=\sum_{\alpha\in S_j}t_j^\alpha$ for $j\in[n]$, whenever $S_j\neq\emptyset$. If $S=(\emptyset,\ldots,\emptyset)$, we represent it as the series $S=0$.
\end{convention}

\begin{convention}
With our convention we can write $S_{\alpha,j}=t_j^\alpha$ instead.
    From now on we will denote $A_{\alpha,j}(c(P))=A_{\alpha,j}(P)=trop_{P}(t_j^\alpha)=trop_{P}(S_{\alpha,j})$. Note that if we write $P=P_1(u_1)\oplus P_2(u_2)\oplus \cdots\oplus P_n(u_n)$,  then  $A_{\alpha,j}(P)=trop_{P}(t_j^\alpha)=trop_{P_j}(t_j^\alpha)=A_{\alpha,j}(P_j)$ for all $j\in[n]$.
\end{convention}

\section{Minimal solutions of a single TLDE}
\label{sect:msstlde}
In this section we will  focus on studying the set $\mu(\Sol(P))$ of minimal non-zero solutions of a single  $P$, where  $P$ is a TLDE of the form \eqref{22_P} in $n\geq1$ unknowns $\{u_1,\ldots,u_n\}$, and of differential order $\mathbf{k}=(k_1,\ldots,k_n)\in(\mathbb{Z}_{>0})^n$. 

We now introduce basic concepts and general properties of the elements of $\mu(\Sol(P))$. 

\begin{remark}
    By our conventions, $P$ has at least two nonzero monomials $a_{0,j}\odot u^{(0)}$ and $a_{1,j}\odot u^{(1)}$, $a_0,a_1\in\mathbb{Z}$. Also, we have $a_{0,j}\neq\infty$ for all $j\in[n]$. These two facts imply that  $trop_P(S)\neq\infty$ if and only if $S\neq(\emptyset,\ldots,\emptyset)$. In what follows we will focus only on nonzero solutions.
\end{remark}


\begin{lemma}\label{lem:minimal_general}
Any minimal  solution $S$ of $P$ has at most two monomials.
\end{lemma}
\begin{proof}
Note that if $S\in \Sol(P)$, then there exists  $i,j\in [n]$ (not necessarily different), $0\leq b\leq k_i$, $0\leq c\leq k_j$, and $t_i^p+t_j^q\leq S$ with $b\leq p$ and $c\leq q$ (again, $p,q$ not necessarily distinct) such that 
    \begin{equation*}
    \label{neq0condition}
        \infty\neq trop_{P}(S)=a_{i,b}+p-b=a_{j,c}+q-c\leq a_{i,j}+\text{val}_{S_j}(i),
    \end{equation*}
    thus $t_i^p+t_j^q$ is also a solution of \eqref{22_P}, and $t_i^p+t_j^q=S$ by minimality.
\end{proof}

It follows from   Lemma \ref{lem:minimal_general} that any $S\in\mu(\Sol(P))$ can be expressed as $S=t_i^p+t_j^q$  with $i,j\in[n]$ (not necessarily distinct) and $p,q\in\mathbb{Z}_{\geq0}$ (not necessarily distinct). From now on we will express the minimal solutions in this form.

\begin{proposition}\label{minimal}
Let $S\in\mu(\Sol(P))$.
\begin{enumerate}[label=\roman*)]
    \item if $S=t_j^p+t_j^q$  where $p<q$, then $p<k_j$;
    \item If $S_1=t_i^p+t_j^{q_1}$ and $S_2=t_i^p+t_j^{q_2}$ are minimal solutions of (\ref{22_P}) satisfying $p<k_i$ and $k_j\leq q_1, q_2$, then $q_1=q_2=q(p)$, which is then uniquely determined.
\end{enumerate}
\end{proposition} 
\begin{proof}
For  $i)$ suppose on the contrary that $k_j\le p$. Then for all  $0\le i\le k_j$ it holds
\begin{eqnarray}\label{2}
a_{i,j}+val_S(i)=a_{i,j}+p-i
\end{eqnarray}
Hence $t_j^p$ is a solution of (\ref{22_P}) which contradicts the minimality of $S=t_j^p+t_j^q$.
    For point $ii)$, suppose that  $q_1<q_2$. Then $q_1$ must satisfy 
    $$trop_P(S_1)=a_{i,l}+p-l=a_{j,j_1}+q_1-j_1$$
    for some $0\leq l\leq p$, and 
    \begin{equation*}
        \begin{cases}
            0\leq j_1\leq k_j\leq q_1,&\text{ if }i\neq j,\\
            p\leq j_1\leq k_j\leq q_1,&\text{ if }i= j.
        \end{cases}
    \end{equation*}
    Likewise, $q_2$ must satisfy 
    $$trop_P(S_2)=a_{i,l'}+p-l'=a_{j,j_1'}+q_2-j_1'$$
    for some $0\leq l'\leq p$, and $0\leq j_1'\leq k_j\leq q_2$ if $i\neq j$, or $p\leq j_1'\leq k_j\leq q_2$ if $i= j$. We claim that $trop_P(S_1)=trop_P(S_2)$, otherwise it wouldn't be the global minimum. But due to (\ref{2}) we have 
    $$a_{j_2}+q_2-j_2>a_{j_2}+q_1-j_2\ge M_1=M_2,$$
    for every $0\leq j_2\leq k_j$ if $i\neq j$, or $p\leq j_2\leq k_j$ if $i= j$, which contradicts that $S_2=t_i^p+t_j^{q_2}$ is a minimal solution.     Thus $S_1=S_2=S=t_i^p+t_j^q$, and $q$ must satisfy $$trop_P(S)=a_{i,l}+p-l=a_{j,j_1}+q-j_1,$$
    so it is uniquely determined by $p$.
\end{proof}

We will now introduce the  decomposition $\mu(\Sol(P))=\mathcal{C}_{\mathbb{Z}_{\geq0}}(P)\sqcup F(P)$ from \eqref{eq:intro_desc} which will be very important in the following.

\begin{definition}
\label{eq:decompo}
    We denote by $\mathcal{C}_{\mathbb{Z}_{\geq0}}(P)$  the set of minimal solutions of $P$ having $t$-adic order at least $\mathbf{k}$; this is $S=(S_1,\ldots,S_n)\in\mu(\Sol(P))$ such that $ord(S_i)\geq k_i$ for all $i\in[n]$. We denote by  $F(P)=\mu(\Sol(P))\setminus\mathcal{C}_{\mathbb{Z}_{\geq0}}(P)$.
\end{definition}

The set $\mathcal{C}_{\mathbb{Z}_{\geq0}}(P)$ will be described in Theorem \ref{thm:s_n=1} for the case $n=1$, and in Theorem \ref{thm:s_n>1} for the case $n>1$.

\begin{remark}
    If $P=\bigoplus_jP_j(u_j)$ for $n>1$, then $\bigcup_{j=1}^n\mu(\Sol(P_j))\subset \mu(\Sol(P))$. Indeed, if $S=S_j(t_j)$ for some $j\in[n]$ (see Convention \ref{conv:rep_sols}), then $trop_P(S)=trop_{P_j}(S)$, so $S=S_j(t_j)\in\Sol(P)$ if and only if $S\in\Sol(P_j)$. Minimality also can be inferred from this fact.  So in order  to understand the case $n>1$, it is necessary to understand first the case $n=1$, which we will discuss in the next section.
\end{remark}

\subsection{The case $n=1$}
In this section we will be exploring the case in which $P$ is a TLDE as in \eqref{1} of fixed differential order $k\in\mathbb{Z}_{>0}$. Recall from Definition \ref{def:linpols_n=1} that  $A_k=\bigoplus_{i=0}^kval_{\{k\}}(i)\odot x_i=min_{\substack{0\leq i\leq k}}\{x_i+(k-i)\}$.


\begin{proposition}
    \label{p:holonomic}
        Let $P$ be a TLDE. The following statements are equivalent:
        \begin{enumerate}[label=\roman*)]
            \item $P$ is non-holonomic,
            \item $S=\{s\}$ is its  solution for some $s\ge k$.
        \end{enumerate}
        In this case $\{q\}$ is also its solution for any $q\ge k$.
    \end{proposition}
\begin{proof}
First assume that $\{s\}$ is a solution of (\ref{1}) for some $s\ge k$. Then the minimum in $\min_{0\le i\le k} \{a_i+s-i\}$ is attained at least twice (cf. (\ref{2})). Therefore the minimum in $\min_{0\le i\le k} \{a_i+q-i\}$ is also attained at least twice for any $q\ge k$. Hence the equation (\ref{1}) is non-holonomic.

Conversely, let (\ref{1}) be non-holonomic. There are at most finite number of minimal solution of (\ref{1}) of the form $\{p,q\},\ p\neq q$ taking into account  Proposition \ref{minimal}. Thus there exists a minimal solution of (\ref{1}) of the form $\{s\}$ for suitable $s\ge k$.
\end{proof}

\begin{remark}\label{singular}
One can easily algorithmically test whether (\ref{1}) is non-holonomic. Due to Proposition~\ref{p:holonomic} ii) the latter is equivalent to  $\{k\}$ being a solution of (\ref{1}), in other words that the minimum in $A_k(P)=\min_{0\le i\le k} \{a_i+k-i\}$ is attained at least twice. Then  $Sing:=V(A_k)\subset \ZZ^{k+1}$ is the set of all non-holonomic equations (\ref{1}). 
\end{remark}

 Let $P$ be a TLDE as in (\ref{1}) of differential order $k$. We are now ready to state the main result of this Section, which is describing the decomposition  $\mu(\Sol(P))=\mathcal{C}_{\mathbb{Z}_{\geq0}}(P)\sqcup F(P)$ from Definition \ref{eq:decompo}.

\begin{theorem}[Structure of $\mu(\Sol(P))$ for $n=1$]
\label{thm:s_n=1} 
Consider $P$ as above.
\begin{enumerate}[label=\roman*)] 
    \item The set $\mathcal{C}_{\mathbb{Z}_{\geq0}}(P)$ satisfies \begin{equation*}
    \mathcal{C}_{\mathbb{Z}_{\geq0}}(P)=\begin{cases}
        \emptyset, &\text{ if }P\notin V(A_k),\\
        \{t^{i}\cdot t^k\::\:i\in \mathbb{Z}_{\geq0}\},&\text{ otherwise}.\\
    \end{cases}
\end{equation*} 
In particular, the system $\Sigma=\{P\}$ is holonomic if and only if $P\notin V(A_k)$.
    \item The set $F(P)$ is finite, and it consists of polynomials $S=t^p+t^q$, with $0\leq p\leq q$ satisfying
    \begin{enumerate}
        \item $0\leq p\leq q< k$, or 
        \item $0\leq p< k\leq q$ for $q=q(p)$ unique (only when $P\notin V(A_k)$).
    \end{enumerate}
\end{enumerate}
\end{theorem}
\begin{proof}
The first part comes from Proposition~\ref{p:holonomic}, and the second part  from Proposition~\ref{minimal}. \end{proof}

\subsubsection{A notion of generic position for ordinary TLDE}
Theorem~\ref{thm:s_n=1} implies that a holonomic equation (\ref{1}) can contain as its minimal solutions sets of the form either $\{j,l\},\ 0\le j\le l<k$ or for each $0\le p<k$ at most one set of the form $\{p,q\}$ where $q\ge k$. 
We call a family of minimal sets of the form  $\{j,l\},\ 0\le j\le l<k$ and of minimal sets of the form $\{p,\star\},\ 0\le p<k$  (where $\star$ stays for some indeterminate $q\ge k$) a {\it configuration}. In this case we use the notation $\{p,q\}\in \{p,\star\}$. Thus, there is a finite number of possible configurations. 

Each configuration $C$ (including the empty one) provides a set $U_C\subset \ZZ^{k+1}$ of holonomic equations $P$ of the form (\ref{1}) satisfying $\mu(Sol(P))=C$.

\begin{theorem}\label{complete}
The collection $\{U_C\}_C$ where $C$ runs through all the configurations, is a partition of $\mathbb{T}^{k+1}\setminus Sing$ into polyhedral complexes in which every polyhedron is given by integer linear constraints of the form either $x(j)-x(l)=c,\ 0\le j,l\le k,\ c\in \ZZ$ or $x(j)-x(l)\ge c$. The polyhedral complex $U_{C_0}$ corresponding to the configuration $C_0:= \{\{i,\star\}\  |\ 0\le i<k\}$  has the full dimension $k+1$.
\end{theorem}
\begin{proof}
First we note that $\dim (Sing)\le k$ because of Remark~\ref{singular}. By the same token if a configuration $C$ contains a minimal solution of the form $\{j\},\ 0\le j<k$ then $\dim U_C \le k$. If $C$ contains a minimal solution of the form $\{j,l\},\ 0\le j\neq l<k$ then it holds  $\min_{0\le i\le j} \{a_i+j-i\} = \min_{0\le i\le l} \{a_i+l-i\}$ (cf. (\ref{2})), and therefore $\dim U_C \le k$ as well.

Now consider the configuration $C_0$. Then the following system of min-plus linear equations holds:
\begin{eqnarray}\label{3}
\min_{0\le i\le j} \{a_i+j-i\} = \min_{j<i\le k} \{a_i+q_j-i\},\ 0\le j<k
\end{eqnarray}
for appropriate $q_j\ge k$ (cf. (\ref{2})). Taking $a_0, \dots, a_k$ satisfying inequalities $a_0\ge a_1\ge \cdots \ge a_k$, we obtain that 
$$\min_{0\le i\le j} \{a_i+j-i\}=a_j,\ \min_{j<i\le k} \{a_i+q_j-i\}= a_k+q_j-k,\ 0\le j<k.$$
\noindent Hence (\ref{3}) implies that $q_j=a_j-a_k+k\ge k,\ 0\le j<k$.

Since for any configuration $C$ such that $C$ contains a minimal solution of the form $\{j,l\},\ 0\le j\le l<k$ it holds $\dim U_C \le k$ (see above), we complete the proof.
\end{proof}

\begin{example}\label{non-holonomic}
    If $k=1$ an equation  $P=a_0\odot u^{(0)}\oplus a_1\odot u^{(1)}=\min \{a_1+u^{(1)},\ a_0+u^{(0)}\}$ as in \eqref{1} has
    \begin{itemize}
        \item no (non-zero) solutions when $a_1-1> a_0$. Below in Section \ref{Section:criterion} we will generalize this result for  arbitrary $k$;
        \item the minimal solutions $\{q\::\:\ q\ge 1\}$ when $a_1-a_0=1$, so the equation (\ref{1}) is non-holonomic;
        \item a single minimal solution  $\{0, a_0-a_1+1\}$ when $a_0-a_1\ge 0$.
    \end{itemize}
Thus, in Theorem~\ref{complete} the polyhedron $U_{C_0}\subset \ZZ^2$ is given by the constraint $a_0-a_1\ge 0$, where $C_0=\{0,\star\}$.
\end{example}

\begin{example}
    While Theorem~\ref{complete} shows that a generic TLDE of order $k$ has $k$ minimal solutions, the (holonomic) equation $P=\bigoplus_{i=0}^k0\odot u^{(i)}=\min_{0\le i\le k} \{u^{(i)}\}$ has $k(k+1)/2$ minimal solutions $\{j,l\},\ 0\le j<l\le k$. Thus, its configuration contains all possible minimal solutions (cf. Proposition~\ref{minimal}).
\end{example}


\begin{definition}
\label{def_reg_n=1}
We say that $P$ is {\bf regular} if 
    \begin{enumerate}
        \item $P\notin \bigcap_{j=0}^kV(A_j)$,  and
        \item $A_0(P)\oplus\cdots\oplus A_{k-1}(P)$ does not vanish weakly in $\mathbb{T}$ (see Definition \ref{def:notvanishstr}).
    \end{enumerate}
\end{definition}

We have an alternative characterization of regularity in terms of the structure of $\mu(\Sol(P))$.

\begin{proposition}
\label{p:Pregn=1}
    We have that $P$ is regular if and only if all of the elements of $\mu(\Sol(P))$ are of the form $\{p,q\}$ with $0\leq p<k\leq q$.
\end{proposition}
\begin{proof}
    The first condition of Definition \ref{def_reg_n=1} is equivalent to the fact that $P$ has no monomial solutions (which implies that it is holonomic). Indeed, note that $A_j(P)=\min_{0\leq i\leq\min\{j,k\}}\{a_{i}+j-i\}$, thus $\{j\}\in\mu(\Sol(P))$ if and only if $P\in V(A_j)$.
    
    The second condition is equivalent to the fact that $P$ has no minimal binomial solutions $\{p,q\}$ with $0\leq p<q<k$. Indeed, for $0\le j_1\neq j_2\leq k-1$ we have $trop_P(\{j_1,j_2\})=trop_P(j_1)\oplus trop_P(j_2)=A_{j_1}(P)\oplus A_{j_2}(P)$ since $trop_P$ is a homomorphism, so $\{j_1, j_2\}$ is a solution of $P$ if $A_{j_1}(P)=A_{j_2}(P)$, and it is minimal since $P$ has already no (minimal) monomial solutions.
\end{proof}

\begin{remark}
\label{rem:trop_disc_n=1}
There exists a tropical algebraic set $\Delta(k)\subset \mathbb{T}^{k+1}$ such that $\Delta(k)\cap\mathbb{Z}^{k+1}$ consists of the TLDEs which are not regular. We call this $\Delta(k)$ the {\it tropical discriminant}. Indeed, the first condition of Definition \ref{def_reg_n=1} is directly a tropical prevariety (intersection of tropical hypersurfaces), and the second condition is also tropical algebraic by  Remark \ref{rem:wv_tropical}.
\end{remark}

\subsubsection{A criterion of existence of a nonzero solution of a univariate TLDE}
\label{Section:criterion}
In the last section we saw that if $P$ as in (\ref{1}) is regular, then $|\mu(Sol(P))|\leq k$. This leads to the question on whether or not a regular $P$ satisfies $\mu(Sol(P))=\emptyset$, which we study here.

For a fixed TLDE $P$, we use the whole sequence of polynomials $\{A_j\::\:j\in\mathbb{Z}_{\geq0}\}$ to define the following map.

\begin{definition}
\label{def:map_A}
  If $P=\bigoplus_{i=0}^k {a_i}\odot u^{(i)}$ we define a map 
\begin{equation*}
    \begin{aligned}
        A_{-}(P):\mathbb{Z}_{\geq0}&\xrightarrow[]{}\mathbb{Z}\\
        j&\mapsto trop_{P}(\{j\})=A_j(P)
    \end{aligned}
\end{equation*}
\end{definition}

The map $A_{-}(P)$ is determined by its values on  the interval $[0,k]\subset \mathbb{Z}_{\geq0}$:

\begin{equation*}
A_{j+1}(P)
    \begin{cases}
        \le A_{j}(P)+1,& 0\leq j<k,\\
        =A_j(P)+1=A_k(P)+j-k,&k\leq j.
    \end{cases}
\end{equation*}

\begin{theorem}
\label{existence}
\begin{enumerate}[label=\roman*)] 
    \item TLDE (\ref{1}) has no solutions (or equivalently, no minimal solutions) if and only if $a_i-i>a_0,\ 1\le i\le k.$
    \item TLDE (\ref{1}) has no minimal solutions of the form $\{j_1, j_2\},\ 0\le j_1\neq j_2$ if and only if $a_i-i\ge a_0,\ 0\le i\le k.$
\end{enumerate}
\end{theorem}
\begin{proof}
If $A_{j_1}(P)=A_{j_2}(P)$ for some $0\le j_1\neq j_2$ then $\{j_1, j_2\}$ is a solution of (\ref{1}).

First assume that (\ref{1}) has no solutions of the form $\{j_1, j_2\}, 0\le j_1\neq j_2$. We claim that $A_{j+1}(P)=A_j(P)+1,\ j\ge 0.$ Suppose the contrary. Take the minimal integer $j_0$ such that $A_{j_0+1}(P)\le A_{j_0}(P),$ clearly $j_0<k$. It holds that $A_j(P)\le A_0(P)=a_0,\ j\ge j_0+1$ due to the supposition. Hence $A_k(P)\le A_0(P)$. Therefore $\{0, k+A_0(P)-A_k(P)\}$ is a solution of (\ref{1}). This justifies the claim.

Thus $A_j(P)=a_0+j,\ j\ge 0$. This implies that $a_i-i\ge a_0,\ 0\le i\le k$. Conversely, if the latter inequalities are valid then $A_j(P)=a_0+j,\ j\ge 0.$  Hence (\ref{1}) has no solutions of the form $0\le j_1\neq j_2$, which completes the proof of ii).

Now assume that (\ref{1}) has no solutions. To justify the strict inequalities in i) suppose the contrary and let $i_0\ge 1$ be the minimal integer such that $a_{i_0}-i_0=a_0$. Then $\{i_0\}$ is a minimal solution of (\ref{1}). This proves i) in one direction: if (\ref{1}) has no solutions then $a_i-i>a_0,\ 1\le i\le k.$
Conversely, assume the latter inequalities. Then $A_j(P)=a_0+j,\ j\ge 0$. No singleton $\{j\}$ is a solution of (\ref{1}) since the minimum in the definition of $A_j$ is attained for a single $i=0$, which proves i).
\end{proof}

\subsubsection{Tropicalization of power series solutions of LDE at $\infty$}

Remind that one can treat a tropical solution of (\ref{1}) as the tropicalization of the initial part of a power series solution $w=\sum_{i\ge 0} b_it^i$ of a classical LDE \eqref{CLDE}. One can view this power series as an expansion in a neighborhood of 0. Alternatively, one can consider a power series $w=\sum_{i\ge 0} c_i t^{-i}$ solution of \eqref{CLDE}
as an expansion in a neighborhood of $\infty$. This leads to the following definition.

\begin{definition}
For a subset $S\subset \ZZ_{\le 0}$ introduce a valuation
\begin{equation*}
    val_S^{(\infty)}(i)=\begin{cases}
        \max \{S\},&\text{ if }i=0,\\
        \max \{s-i\ |\ -1\ge s\in S\},&\text{ if } 1\le i\le k.\\
    \end{cases}
\end{equation*}

\noindent We say that $S$ is a {\it tropical solution at $\infty$} of (\ref{1}) if the maximum in 
$$\max_{0\le i\le k} \{a_i+val_S^{(\infty)}(i)\}$$
is attained at least for two values of $i$.
\end{definition}

\begin{lemma}\label{infty}
A minimal tropical solution $S$ at $\infty$ of (\ref{1}) can be one of two following types:
\begin{enumerate}[label=\roman*)]
    \item $S=\{-r\}$ for an arbitrary $r\ge 1$ iff the maximum in $\max_{0\le i\le k} \{a_i-i\}$ is attained at least twice;
    \item $S=\{0, -r\}$ iff $1\le r:= \max_{0\le i\le k} \{a_i-i\}-a_0$.
\end{enumerate}
\end{lemma}
\begin{proof}
    Similar to Proposition~\ref{minimal} a minimal solution $S\subset \ZZ_{\le 0}$ at $\infty$ consists of at most of two elements. If $0\notin S$ then $S=\{-r\}$ for an arbitrary $r\ge 1$. In this case  the maximum in $\max_{0\le i\le k} \{a_i-i\}$ is attained at least twice. The converse is also true.

Now let $0\in S$. Clearly, $\{0\}$ is not a solution at $\infty$ of (\ref{1}). Therefore $S=\{0, -r\},\ r\ge 1$ and $a_0=\max_{0\le i\le k} \{a_i-i\}-r$. Conversely, if $r:= \max_{0\le i\le k} \{a_i-i\}-a_0 \ge 1$ then $\{0, -r\}$ is a minimal solution at $\infty$ of (\ref{1}).
\end{proof}

\begin{corollary}
\begin{enumerate}[label=\roman*)]
    \item TLDE (\ref{1}) has either a tropical solution or a tropical solution at $\infty$. 

\item If (\ref{1}) has only minimal tropical solutions being singletons then (\ref{1}) has a minimal tropical solution at $\infty$ of the form $\{-r\}$ for any $r\ge 1$.
\end{enumerate}
\end{corollary}
\begin{proof}
    \begin{enumerate}[label=\roman*)]
    \item Let (\ref{1}) have no tropical solutions. Then $a_i-i>a_0,\ 1\le i\le k$ due to Theorem~\ref{existence}~i). Hence (\ref{1}) has a (minimal) tropical solution $\{0, a_0-\max_{0\le i\le k} \{a_i-i\}\}$ at $\infty$ because of Lemma~\ref{infty}~ii).
    \item Due to Theorem~\ref{existence} it holds $a_i-i\ge a_0,\ 0\le i\le k$, and there exists $1\le i\le k$ such that $a_i-i=a_0$. Therefore (\ref{1}) has a minimal tropical solution $\{-r\}$ at $\infty$ for an arbitrary $r\ge 1$ because of Lemma~\ref{infty}~i).\qedhere
    \end{enumerate}
\end{proof}

\subsection{The  case $n>1$}
In this Section we assume that $n>1$ and $P$ is as in \eqref{22_P} and has  fixed differential order $\mathbf{k}=(k_1,\ldots,k_n)\in(\mathbb{Z}_{>0})^n$. As usual, we write $P=\bigoplus_jP_j(u_j)$ with $P_j(u_j)$ of differential order $k_j\in \mathbb{Z}_{>0}$ for $j\in[n]$. 

The main result of this Section is a complete description of the parts $\mu(\Sol(P))=\mathcal{C}_{\mathbb{Z}_{\geq0}}(P)\sqcup F(P)$ from Definition \ref{eq:decompo}; this is a generalization of Theorem \ref{thm:s_n=1}.

\begin{definition}
   If $t_i^p+t_j^q=S\in\mu(\Sol(P))$ (see Lemma \ref{lem:minimal_general}), we define its support $Supp(S)=\{i,j\}\subset [n]$. 
\end{definition}  

\begin{theorem}\label{thm:form_finite_part}
       For $P$  as before, the set $F(P)$ is finite. Furthermore, if $S\in F(P)$, then either 
\begin{enumerate}
\item $Supp(S)=\{j\}$ if and only if  $S\in F(P_j)$,
\item $Supp(S)=\{i,j\}$ if and only if  $S=t_i^p+t_j^q$ satisfies
\begin{enumerate}
    \item $(p,q)<(k_i,k_j)$,
    \item $p<k_i$ and $q=q(p)\geq k_j$ is unique.
\end{enumerate}
\end{enumerate}      
  \end{theorem}    
   \begin{proof}
    It is clear that $Supp(S)=\{j\}$ iff $S\in F(P_j)$, and $\bigcup_jF(P_j)$ is finite by Theorem \ref{thm:s_n=1}. If $Supp(S)=\{i,j\}$ with $\varphi=t_i^p+t_j^q$, then we have a trichotomy:
\begin{enumerate}
    \item $(p,q)<(k_i,k_j)$, 
    \item $p<k_i$ and $q\geq k_j$, 
    \item $(p,q)\geq (k_i,k_j)$,
\end{enumerate} 
and we are ruling out the third possibility. So we just need to show that there is a finite number of  possibilities (2), but this is just Proposition \ref{minimal} ii).
\end{proof} 

We will study the set $\mathcal{C}_{\mathbb{Z}_{\geq0}}(P)$ for the case $n>1$ in the next section.  Contrary to the case $n=1$, we will see that it is always non-empty (it is in fact, infinite).

\subsubsection{Structure of the set $\mathcal{C}_{\mathbb{Z}_{\geq0}}(P)$}
Let $n>1$ and $P=P_1(u_1)\oplus P_2(u_2)\oplus \cdots\oplus P_n(u_n)$. The main result of this Section is a complete characterization of  $\mathcal{C}_{\mathbb{Z}_{\geq0}}(P)$. 

 Consider the order map $ord:\mathcal{P}(\mathbb{Z}_{\geq0})^n\xrightarrow[]{}\mathbb{T}^n$ sending $S=(S_1,\ldots,S_n)$ to $ord(S)=(ord(S_1),\ldots,ord(S_n))$. This is an (order-reversing) homomorphism of $\mathbb{B}$-semimodules  that induces a bijection of sets 
\begin{equation}
    \mathcal{P}(\mathbb{Z}_{\geq0})^n\supset\mathcal{C}_{\mathbb{Z}_{\geq0}}(P)\xrightarrow[]{\sim}ord(\mathcal{C}_{\mathbb{Z}_{\geq0}}(P))\subset (\mathbb{Z}\cup\{\infty\})^n
\end{equation}

Indeed, we know that if $Supp(S)=\{i\}$, then $S=t_i^p$ is a monomial by Theorem \ref{thm:s_n=1}, so $ord(S)=(\infty,\ldots,p,\ldots,\infty)$ for some $p\geq k_i$, and if $Supp(S)=\{i,j\}$, then $S=t_i^p+t_j^q$, and $ord(S)=(\infty,\ldots,p,\ldots,q,\ldots,\infty)$ for some $p\geq k_i$ and $q\geq k_j$.

For the rest of this section we will use this identification.

\begin{theorem}
\label{thm_val_mat}
    Let $P$ be a TLDE with $n>1$. 
    The set $\mathcal{C}_{\mathbb{Z}_{\geq0}}(P)$ satisfies
    \begin{enumerate}[label=\roman*)]
        \item $(\infty,\ldots,\infty)\notin \mathcal{C}_{\mathbb{Z}_{\geq0}}(P)$,
        \item if $S\in\mathcal{C}_{\mathbb{Z}_{\geq0}}(P)$ and $i\in \mathbb{Z}_{\geq0}$, then $i\odot S\in \mathcal{C}_{\mathbb{Z}_{\geq0}}(P)$,
        \item if $S,T\in\mathcal{C}_{\mathbb{Z}_{\geq0}}(P)$ with $Supp(S)\neq Supp(T)$, then $Supp(S)\not\subset Supp(T)$,
        \item if $S,T\in\mathcal{C}_{\mathbb{Z}_{\geq0}}(P)$ and $i,j\in [n]$ with $s_i=t_i\neq\infty$ and $s_j<t_j$, there exists $U\in \mathcal{C}_{\mathbb{Z}_{\geq0}}(P)$ such that $u_i=\infty$, $u_j=s_j$ and $U\geq \min\{S,T\}$.
    \end{enumerate}
\end{theorem}
\begin{proof}
Point i) is equivalent to $(\emptyset,\ldots,\emptyset)\notin\mathcal{C}_{\mathbb{Z}_{\geq0}}(P)$. For point ii), if $S\in\mathcal{C}_{\mathbb{Z}_{\geq0}}(P)$, then $i\odot S$ for $i\in \mathbb{Z}_{\geq0}$ represents $t^i\cdot S$, which is once again in $\mathcal{C}_{\mathbb{Z}_{\geq0}}(P)$. Point iii) follows from minimality. 

    The proof of  iv) is similar to Proposition \ref{minimal} ii), which uses (\ref{2}): there exist $S=t_i^p+t_j^q$ with $p,q<\infty$ and $T=t_i^p+t_j^r$ with $r\in[q+1,\infty]$, but this is not possible, since  $r=\infty$ contradicts the minimality of $S$, and if $r\neq\infty$, we have 
    \begin{eqnarray*}
        a_{j,s}+q-s=trop_P(S)=a_{i,l}+p-l=trop_P(T)=a_{j,t}+r-t
    \end{eqnarray*}
    but $k_j\leq q<r$, and $a_{j,t}+q-t<a_{j,s}+q-s$, which contradicts the fact that $S$ is a solution. Thus, for any given  $S=t_i^p+t_j^q$, such $T=t_i^p+t_j^r$ does not exist.
\end{proof}

\begin{remark}
    In Theorem \ref{thm_val_mat}, if in point ii) we change ``$i\in \mathbb{Z}_{\geq0}$'' by ``$i\in \mathbb{R}$'', then the resulting object $\mathcal{C}_{\mathbb{R}}(P)$ is a collection of valuated circuits, and the pair $([n],\mathcal{C}_{\mathbb{R}}(P))$ is then called a valuated matroid (or tropical linear space). Thus Theorem \ref{thm_val_mat} says then that  $\mathcal{C}_{\mathbb{Z}_{\geq0}}(P)$  is contained in the valuated matroid $([n],\mathcal{C}_{\mathbb{R}}(P))$.
\end{remark} 
    
By point ii),  the monoid $(\mathbb{Z}_{\geq0},\odot,0)$ acts on $\mathcal{C}_{\mathbb{Z}_{\geq0}}(P)$. We will see below that the set $\mathcal{C}_{\mathbb{Z}_{\geq0}}(P)$ can be recovered from the quotient set $\mathcal{C}_{\mathbb{Z}_{\geq0}}(P)/\mathbb{Z}_{\geq0}$.

We denote by $\mathcal{C}(P)\in\mathcal{P}([n])$  the family of supports of elements of $\mathcal{C}_{\mathbb{Z}_{\geq0}}(P)$, and by $L(P):=\{j\in [n]\::\:P_j\in V(A_{k_j})\}$.
\begin{convention}
    Even if $A_{k_j,j}$ is a tropical linear polynomial for $j\in [n]$ (see Definitions \ref{def:linpols_n>1}, \ref{def:map_A}), sometimes we will denote the common value $A_{k_j,j}(P)=A_{k_j,j}(P_j)=\min_{0\leq i\leq k_j}\{a_{i,j}+k_j-i\}$ just by $A_{k_j,j}$ to simplify the notation.
\end{convention}

\begin{corollary}
\label{eq:circuits_n>1}
    The pair $M(P)=([n],\mathcal{C}(P))$ is a matroid (of circuits). Moreover, we have $\mathcal{C}(P):=L(P)\cup\binom{[n]\setminus L(P)}{2}$, 
thus $L(P)$ is the set of loops of  $M(P)$.
\end{corollary}
\begin{proof}
By the previous remark, we have that $M(P)=([n],\mathcal{C}(P))$ is the underlying matroid of the valuated matroid $([n],\mathcal{C}_{\mathbb{R}}(P))$. It is also easy to see that $L(P)\cup\binom{[n]\setminus L(P)}{2}$ satisfies the axioms for circuits (once we prove it).

Let $S\in \mathcal{C}_{\mathbb{Z}_{\geq0}}(P)$ such that $Supp(S)=\{j\}\in\mathcal{C}$, then there exists $r\geq k_j$ such that $S=t_j^{r}\in\Sol(P_j)$, but this is equivalent to $P_j\in V(A_{k_j,j})$ by Proposition \ref{p:holonomic}. 

We now consider $\binom{[n]\setminus L(P)}{2}$, which is empty if $|[n]\setminus L(P)|\leq1$. If $[n]\setminus L(P)=\emptyset$, there is nothing left to do.

If $[n]\setminus L(P)=\{j\}$, then $j$ can not belong to a pair $\{i,j\}$, otherwise we would have $i\in L(P)$,  and by this reason there is no minimal solution $S=t_i^p+t_j^q$ with $p\geq k_i$ and $q\geq k_j$.

So we suppose that $|[n]\setminus L(P)|>1$, and we need to show that if $a,b\in [n]\setminus L(P)$, then there exists $S\in \mathcal{C}_{\mathbb{Z}_{\geq0}}(P)$ such that $Supp(S)=\{a,b\}$. 

W.L.O.G. we can suppose that $A_{k_a,a}-A_{k_b,b}\geq0$ (otherwise we choose $A_{k_b,b}-A_{k_a,a}\geq0$), then $S=t_{a}^{k_a}+t_b^{k_b+A_{k_a,a}-A_{k_b,b}}\in \mathcal{C}_{\mathbb{Z}_{\geq0}}(P)$ satisfies $Supp(S)=\{a,b\}.$
\end{proof}

We are now ready to prove the main result of this section.

\begin{theorem}[Structure of $\mu(\Sol(P))$ for $n>1$]
\label{thm:s_n>1} 
Consider $P=\bigoplus_{j=1}^nP_j$ as before, and $L(P):=\{j\in [n]\::\:P_j\in V(A_{k_j})\}$. Then 
\begin{enumerate}[label=\roman*)] 
    \item the set $\mathcal{C}_{\mathbb{Z}_{\geq0}}(P)$ satisfies  \begin{equation*}
        \mathcal{C}_{\mathbb{Z}_{\geq0}}(P)=\bigcup_{j\in L(P)}\mathbb{Z}_{\geq0}\odot \{k_j\}\cup\bigcup_{a,b\notin L(P)\::\:A_{k_b,b}\leq A_{k_a,a}}\mathbb{Z}_{\geq0}\odot\{(k_a,k_b+A_{k_a,a}-A_{k_b,b})\}.
    \end{equation*}
    \item The set $F(P):=\mu(\Sol(P))\setminus\mathcal{C}_{\mathbb{Z}_{\geq0}}(P)$ is finite. Furthermore, if $S\in F(P)$ with $S=Supp(S)$, then either 
\begin{enumerate}
\item $S=\{j\}$ iff  $S\in F(P_j)$,
\item $S=\{i,j\}$ iff $S=t_i^p+t_j^q$ satisfies
\begin{enumerate}
    \item $(p,q)<(k_i,k_j)$,
    \item if $p<k_i$, then $j\notin L(P)$ and $q=q(p)\geq k_j$ is unique (or $i\notin L(P)$ and $q< k_j$ and $p=p(q)\geq k_i$ is unique), 
\end{enumerate}
\end{enumerate}
\end{enumerate}
\end{theorem}
\begin{proof}
The part ii) is Theorem~\ref{thm:form_finite_part}. For part i) it suffices to show the inclusion $\subseteq$. Let $S\in \mathcal{C}_{\mathbb{Z}_{\geq0}}(P)$. If $Supp(S)=\{i\}\in L(P)$, then there exists $s\in\mathbb{Z}_{\geq0}$ such that  $S=t^s\cdot t_i^{k_i}$, so $\{S\in \mathcal{C}_{\mathbb{Z}_{\geq0}}(P)\::\:Supp(S)=\{i\}\}=\mathbb{Z}_{\geq0}\cdot \{k_i\}$.

If $Supp(S)=\{a,b\}\in \binom{[n]\setminus L(M(P))}{2}$, then $S=t_a^{k_a+\alpha}+t_b^{k_b+\beta}$ for some $\alpha,\beta\in\mathbb{Z}_{\geq0}$. This means 
     \begin{equation*}
        trop_P(S)= A_{P_a}(k_a)+\alpha=A_{P_b}(k_b)+\beta,
     \end{equation*}
     W.L.O.G. we can suppose that $A_{k_a,a}-A_{k_b,b}=\beta-\alpha\geq0$ (otherwise we choose $A_{k_b,b}-A_{k_a,a}\geq0$), so that $A_{k_a,a}=A_{k_b,b}+(\beta-\alpha)$ and $S=t^\alpha\cdot (t_{a}^{k_a}+t_b^{k_b+\nu(\{a,b\})})$. So $\{S\in \mathcal{C}_{\mathbb{Z}_{\geq0}}(P)\::\:Supp(S)=\{a,b\}\}=\mathbb{Z}_{\geq0}\cdot (t_{a}^{k_a}+t_b^{k_b+A_{k_a,a}-A_{k_b,b}})$.
\end{proof}

The action of the semigroup $\mathbb{Z}_{\geq0}$ on $\mathcal{C}_{\mathbb{Z}_{\geq0}}(P)$ becomes clear from Theorem \ref{thm:s_n>1}. 

\begin{corollary}
\label{cor:ordergeqkn>1_quot}
    The quotient set $\mathcal{C}_{\mathbb{Z}_{\geq0}}(P)/\mathbb{Z}_{\geq0}$ is determined by a {\it valuation} of circuits, which is the map $\nu:\mathcal{C}(P)\xrightarrow[]{}\mathbb{Z}_{\geq0}$  given by 
\begin{equation*}
\nu(S)=\begin{cases}
    k_j,&\text{ if }j\in L(P),\\
    A_{k_a,a}-A_{k_b,b},&\text{ if }\{a,b\}\in \binom{[n]\setminus L(P)}{2},A_{k_b,b}\leq A_{k_a,a}.\\
\end{cases}
\end{equation*}
And the set $\mathcal{C}_{\mathbb{Z}_{\geq0}}(P)$ can be recovered from $\mathcal{C}_{\mathbb{Z}_{\geq0}}(P)/\mathbb{Z}_{\geq0}$.
\end{corollary}

  The next step is to introduce the concept of regularity for polynomials $P=\bigoplus_{j=1}^nP_j$ for $n>1$, to explain and simplify Definition \ref{regular}. To do so, we will use the family $\{A_{\alpha,j}\::\:j\in [n],\:0\leq \alpha\leq k_j\}$ from Definition \ref{def:linpols_n>1}.

\begin{definition}
\label{def_reg_n>1}
    We say that $P=\bigoplus_{j=1}^nP_j(u_j)$ is {\bf regular} if 
    \begin{enumerate}
        \item (local condition)  $P_j$ is regular for all $j\in[n]$ as in Definition \ref{def_reg_n=1}, 
        \item (global condition) the expression $\bigoplus_{j=1}^n\bigoplus_{\alpha=0}^{k_j-1}A_{\alpha,j}(P)$ does not vanish weakly in $\mathbb{T}$ (see Definition \ref{def:notvanishstr}).
    \end{enumerate}
\end{definition}

Similarly to Proposition \ref{p:Pregn=1}, we have the following characterization of regularity in terms of controlling the elements of $F(P)$.

\begin{proposition}
\label{p:Pregn>1}
    We have that $P$ is regular if and only if the elements of $F(P)$ from the decomposition \eqref{eq:intro_desc} are of the form 
    \begin{enumerate}
        \item $t_j^p+t_j^q$ with $0\leq p<k_j\leq q$,
        \item $t_i^p+t_j^q$ with $0\leq p<k_i$ and $k_j\leq q$.
    \end{enumerate}
    
\end{proposition}
\begin{proof}
    The first condition of Definition \ref{def_reg_n>1} controls  $\bigcup_jF(P_j)$; the second one prevents the existence of solutions  $S=t_i^p+t_j^q$ with $(p,q)<(k_i,k_j)$, which, according to Theorem \ref{thm:form_finite_part} are the only solutions of $F(P)$ left to be controlled.\qedhere
\end{proof}

If $P=\bigoplus_{j=1}^nP_j(u_j)$ is regular, then $L(P)=\emptyset$ since each $P_j$ is holonomic, and according to Corollary \ref{cor:ordergeqkn>1_quot}, the set $\mathcal{C}_{\mathbb{Z}_{\geq0}}(P)$ becomes controlled by the values  $A_{k_j,j}(P)-A_{i,k_i}(P)$ for $\{i,j\}\in \binom{[n]}{2}$. 

\begin{remark}
    \label{rem:trop_disc_n>1}
There exists a tropical algebraic set $\Delta(\mathbf{k})\subset \prod_{j=1}^n\mathbb{T}^{k_j+1}$ such that $\Delta(\mathbf{k})\cap\prod_{j=1}^n\mathbb{Z}^{k_j+1}$ consists of the TLDEs which are not regular. We call this $\Delta(\mathbf{k})$ the {\it tropical discriminant}. Indeed, both conditions of Definition \ref{def_reg_n>1} are tropical algebraic, the first one by Remark \ref{rem:trop_disc_n=1}, and the second one by Remark \ref{rem:wv_tropical}.
\end{remark}

   \section{Systems of tropical ordinary linear differential equations}\label{four}
In this section we define the important concepts of  regularity and genericity for holonomic systems of TLDEs. See Definition \ref{regular}.
   
Consider a system $\Sigma=\{P_1,\ldots,P_m\}$ of $m$ tropical ordinary linear differential equations in $n\geq1$ unknowns (cf. (\ref{22_P})), each one of differential order  $\mathbf{k}=(k_1,\ldots,k_n)\in(\mathbb{Z}_{>0})^n$:
\begin{equation}
\label{4M}
P_l=\bigoplus_{1\le j\le n,\ 0\le i\le k_{j}}a_{ijl}\odot u_j^{(i)}=\min_{1\le j\le n,\ 0\le i\le k_{j}} \{a_{ijl}+u_j^{(i)}\},\ 1\le l\le m.
\end{equation}

Recall that the set of solutions of  the system \eqref{4M} is the $\mathbb{B}$-subsemimodule $\Sol(\Sigma)=\bigcap_j\Sol(P_j)$ of $\mathcal{P}(\mathbb{Z}_{\geq0})$, and if  $(S_1,\ldots,S_n)\in \Sol(\Sigma)$ we have that 

\begin{eqnarray}\label{5}
\min_{1\le j\le n,\ 0\le i\le k_{j}} \{a_{ijl}+ val_{S_j}(i)\}
\end{eqnarray}
is attained at least twice for each $1\le l\le m$.

We have that $\Sol(\Sigma)$ is  generated by  $\mu(\Sol(\Sigma))$, and that $\Sigma$ is holonomic if $|\mu(\Sol(\Sigma))|<\infty$. Similar to Proposition~\ref{minimal} one can show that if $S=(S_1,\ldots,S_n)\in \mu(\Sol(\Sigma))$, then  $S_j \cap \ZZ_{\ge k_j}$ contains at most one element  for $1\le j\le n$.

\begin{proposition}\label{several}
\begin{enumerate}[label=\roman*)]
\item When $m<n$ a system (\ref{4}) is non-holonomic;
\item When $m=n$ a generic system (\ref{4}) is holonomic.
\end{enumerate}
\end{proposition}
\begin{proof}
    \begin{enumerate}[label=\roman*)]
\item Take $S_j:=\{q_j\},\ 1\le j\le n$ for unknowns $q_j\in \ZZ_{\ge k_j}$. Then one can rewrite (\ref{5}) as a system of $m$ tropical homogeneous linear equations in $n$ unknowns. This system has a solution of the form $(q_1+b, \dots, q_n+b)$ for suitable $q_j \in \ZZ,\ 1\le j\le n$ and for an arbitrary $b\in \ZZ$ (see e.g. \cite{Butkovic}, \cite{AGG}, \cite{G13}).
\item For each $1\le j\le n$ fix a subset $S_j \cap \{0, \dots, k_j-1\}$ and 
$S_j\cap \ZZ_{\ge k_j}= \{q_j\}$ for an unknown $q_j$. One can rewrite (\ref{5}) as a system of $n$ tropical (non-homogeneous) linear equations in $n$ unknowns. In general, such a system has a unique solution due to the Tropical Bezout Theorem \cite{MS}.\qedhere
    \end{enumerate}
\end{proof}

\begin{remark}
Now we show that one can algorithmically test whether a system (\ref{4}) is holonomic. To this end one fixes $S_j \cap \{0, \dots, k_j-1\},\ 1\le j\le n$ and verifies whether a system of tropical linear equations with unknowns $q_j,\ 1\le j\le n$ (as in the proof of Proposition~\ref{several}~ii)) has a finite number of solutions in $\ZZ_{\ge 0}^n$. If this system is homogeneous one has to check its solvability in $\ZZ ^n$ (say, based on one of the algorithms from \cite{Butkovic}, \cite{AGG}, \cite{G13}). If the system is solvable then the system (\ref{4}) is non-holonomic (see the proof of  Proposition~\ref{several}~ii)). 

Now assume that this tropical linear system is not necessary homogeneous. For a non-homogeneous equation of this system of the form
\begin{eqnarray}\label{6}
\min_{1\le j\le n} \{b_j+q_j, b\},\ b\neq \infty
\end{eqnarray} 
take either a pair $1\le j_1<j_2\le n$ or a singleton $1\le j_0\le n$, respectively. We treat them as candidates for minima: i.e. either $b_{j_1}+q_{j_1}=b_{j_2}+q_{j_2}=\min_{1\le j\le n} \{b_j+q_j, b\}$ or $b_{j_0}+q_{j_0}=b=\min_{1\le j\le n} \{b_j+q_j, b\}$, respectively. In the former case it holds $k_{j_1}\le q_{j_1}\le b-b_{j_1},\  k_{j_2}\le q_{j_2}\le b-b_{j_2}$, and we have a finite number of possibilities for $q_{j_1}, q_{j_2}$. In the latter case it holds $q_{j_0}=b-b_{j_0}$. For each choice of $q_{j_1}, q_{j_2}$ or $q_{j_0}$, respectively, we proceed to the next equation of the form (\ref{6}) (substituting the chosen values of  $q_{j_1}, q_{j_2}$ or $q_{j_0}$, respectively).

After exhausting all the equations of the form (\ref{6})  take the truncated equation for each considered equation of the form (\ref{6}) by removing from it all non-assigned unknowns $q_j$. The algorithm checks whether the truncated tropical system is satisfied. If it is not satisfied, then the algorithm proceeds to another its branch by choosing different  $q_{j_1}, q_{j_2}$ or $q_{j_0}$ and their values. If after exhausting all the equations of the form (\ref{6}),  some homogeneous (i.e. with $b=\infty$) linear equations still remain, the algorithm checks the solvability of this homogeneous system (cf. above). If this homogeneous system has a solution then the initial system (\ref{4}) is non-holonomic (as above). If this homogeneous system has no solution then the algorithm proceeds to another its  branch. 

Finally, if after exhausting all the equations of the form (\ref{6}) no homogeneous equations remain, and at least one  non-assigned unknown $q_j$ remains, then one can take the values of all the latter unknowns as arbitrary large integers. Therefore, in this case the initial system (\ref{4}) is non-holonomic. Thus, if no branch of the algorithm yields the output that the initial system (\ref{4}) is non-holonomic, the algorithm concludes that  (\ref{4}) is holonomic.

Unfortunately, the complexity bound of the described algorithm is exponential.
\end{remark}

Consider  a holonomic system $\Sigma=\{P_1,\ldots,P_n\}$ on  $n>1$ unknowns as in \eqref{4}. We introduce two different notions of general position for $\Sigma$, namely regular system, based on the the concept of regular polynomial in $n>1$ unknowns from Definition \ref{def_reg_n>1}, and generic system,  using also the tropical polynomials $A_{\alpha,j}$ for $j\in [n]$ and $0\leq \alpha\leq k_j$ from Definition \ref{def:linpols_n>1}.


\begin{definition}\label{regular}
Consider a system (\ref{4}).
\begin{enumerate}
    \item We say that it is {\bf regular} if  $P_l$ is regular in the sense of Definition \ref{def_reg_n>1} for all $l\in[n]$.
    \item We say that it is {\bf generic} if the expressions
\begin{equation}\label{61}
A_{k_{\sigma(1)},\sigma(1)}(P_1)\odot\cdots \odot A_{k_{\sigma(n)},\sigma(n)}(P_n)=\sum_{1\le j\le n}A_{k_{\sigma(j)},\sigma(j)}(P_j)
\end{equation}
are pairwise distinct for all  $\sigma\in Sym(n)$, and in addition, for $l_1\neq l_2$ it holds
\begin{equation*}\label{62}
    A_{k_j,j}(P_{l_1})-A_{\alpha,j_1}(P_{l_1})\neq A_{k_j,j}(P_{l_2})-A_{\beta,j_2}(P_{l_2}),\ 0\le \alpha<k_{j_1},\ 0\le \beta<k_{j_2}.
\end{equation*}
\end{enumerate}
\end{definition}

Note that, for fixed  $\mathbf{k}=(k_1,\ldots,k_n)$,  the set of systems (\ref{4}) being not generic or not regular, lie  in a finite union of polyhedra of dimensions less than the full dimension $n(k_1+\cdots+k_n)$ inside of the space of all systems (\ref{4}). 

\begin{remark}
    Moreover, it follows from Remark \ref{rem:trop_disc_n>1} that the set of non-regular systems (\ref{4}) is a tropical algebraic set. The same is true for the set of non-generic systems (\ref{4}), indeed, the first condition of genericity \eqref{61}  is equivalent to the fact that the {\it tropical determinant}  of the matrix $(A_{k_j,j}(P_l))_{1\leq l,j\leq n}$ does not vanish weakly (see Definition \ref{def:notvanishstr}). The second condition of  can be reformulated as asking that the tropical determinant $tdet\begin{pmatrix} 
A_{k_j,j}(P_{l_1}) & A_{\alpha,j_1}(P_{l_1}) \\
A_{k_j,j}(P_{l_2}) & A_{\beta,j_2}(P_{l_2}) 
\end{pmatrix}$ does  not weakly vanishing.
\end{remark}

\begin{convention}
Given $1\leq j,l\leq n$ and $0\leq \alpha \leq k_j$, for the computations that follow we will write $A_{\alpha,j}(P_l)=A_{\alpha jl}$. Note that if $P_l=\bigoplus_{j=1}^nP_{j,l}(u_j)$ for $l\in[n]$, then $A_{\alpha,j}(P_l)=A_{\alpha,j}(P_{j,l})$ for every $j\in[n]$ and $0\leq \alpha \leq k_j$.
\end{convention}

\subsection{Upper bound on the number of minimal solutions for  $m=n=2$}
In this section we consider a generic regular  system (\ref{4}) for $n=m=2$ with unknowns $u,v$, this is $P_1=P_{u,1}(u)+P_{v,1}(v)$ and $P_2=P_{u,2}(u)+P_{v,2}(v)$ each one of differential order $\mathbf{k}=(k_u,k_v)$.  Note that 
\begin{eqnarray}\label{29}
A_{i_1jl}=A_{jl}+i_1-k_j,\ i_1\ge k_j,\ 1\le j,l\le n.
\end{eqnarray}

and also \begin{eqnarray}\label{57}
A_{0,1,l},\dots, A_{k_1-1,1,l}, A_{0,2,l},\dots, A_{k_2-1,2,l},\dots, A_{0,n,l},\dots,A_{k_n-1,n,l};
\end{eqnarray}
are pairwise distinct.

 For a minimal solution $(S_u, S_v)$ of (\ref{4}) and $l=1, 2$ denote by $S_{ul}\subset S_u,\ S_{vl}\subset S_v$ the subsets such that the minimum in (\ref{5}) is attained at $S_{ul} \cup S_{vl}$. Then $S_u=S_{u1} \cup S_{u2},\ S_v=S_{v1}\cup S_{v2}$. If $i$ belongs to $S_{u1} \cap S_{u2}$ then we say that $i$ belongs to $S_u$ with the multiplicity 2.

 Due to minimality of $(S_u, S_v)$ it holds $|S_u|+|S_v|\le 4$. Moreover, one can assume that $|S_{ul} \cap \{0,\dots, k_u-1\}| + |S_{vl} \cap \{0,\dots, k_v-1\}|\le 1,\ l=1,2$, since otherwise, the system (\ref{4}) is not regular (cf. Theorem~\ref{complete}). In addition, $|S_u|+|S_v|\ge 3$ since otherwise, the system (\ref{4}) is again not regular.  Now we describe all possible candidates for minimal solutions of a generic regular system (\ref{4}) (cf. Proposition~\ref{minimal}). 

 \begin{lemma}\label{unique}
Any minimal solution of a generic regular system (\ref{4}) is one of the following 5 types of configurations:

$i)_{uv}$ $S_u\in \{i, \star\},\ S_v\in \{p,\star\},\ 0\le i<k_u, 0\le p<k_v$; 

$i)_{uu}$ $S_u\in \{i,i_0,\star\},\ S_v\in \{\star\},\ 0\le i\le i_0<k_u$;

$i)_{vv}$ $S_u\in \{\star\},\ S_v\in \{p,p_0,\star\},\ 0\le p\le p_0<k_v$;

$i)_u$ $S_u\in \{i,\star\},\ S_v\in \{\star\}$;

$i)_v$ $S_u\in \{\star\},\ S_v\in \{p, \star\}$.

For each $i,i_0,p,p_0$ (as above) there exist at most one integer $q_u\ge k_u$ and at most one integer $q_v\ge k_v$ such that $q_u\in S_u,\ q_v\in S_v$ for all 5 types, respectively, provided that $(S_u, S_v)$ is a minimal solution of (\ref{4}). Moreover, the subsets $S_{u1}, S_{u2}, S_{v1}, S_{v2}$ are uniquely determined by $i,i_0,p,p_0$ for all 5 types, respectively, provided that $(S_u, S_v)$ is a minimal solution of (\ref{4}). 
\end{lemma}
\begin{proof}
    Denote (cf.  Definition~\ref{regular}~ii)) $A:=A_{u1}-A_{u2}-A_{v1}+A_{v2}$. Due to genericity of (\ref{4}), see (\ref{61}) it holds $A\neq 0$.

    $i)_{uv}$. Let  $S_u=\{i,q_u\},\ S_v=\{p,q_v\},\ 0\le i<k_u, 0\le p<k_v$ be a minimal solution of (\ref{4}). There are 4 cases (taking into account (\ref{62})):

\begin{align}
    A_{iu1}&=A_{q_uu1},\ A_{pv2}=A_{q_vv2}; \label{7}\\
    A_{iu1}&=A_{q_vv1},\ A_{pv2}=A_{q_uu2};\label{8}\\
A_{iu2}&=A_{q_uu2},\ A_{pv1}=A_{q_vv1}; \label{9}\\
A_{iu2}&=A_{q_vv2},\ A_{pv1}=A_{q_uu1}.\label{10}
\end{align}

We claim that exactly one of (\ref{7})-(\ref{10}) is valid. W.l.o.g. assume that (\ref{7}) is valid (other three cases can be studied in a similar way). Then (\ref{29}) implies that
\begin{eqnarray}\label{11}
A_{iu1}=A_{u1}+q_u-k_u,\ A_{pv2}=A_{v2}+q_v-k_v.    
\end{eqnarray}
This provides the expressions for $q_u, q_v$. In addition, we have

\begin{align}
    A_{iu1}\le A_{q_vv1}&,\ A_{pv2}\le A_{q_uu2}; \label{12}\\
A_{iu1}<A_{pv1}&,\ A_{pv2}<A_{iu2}\label{13}
\end{align}

due to (\ref{57}).

Observe that (\ref{13}) contradicts to $A_{iu2}=A_{q_u^{(1)}u2},\ A_{pv1}=A_{q_v^{(1)}v1}$ (see (\ref{9})) with any $q_u^{(1)}\ge k_u, q_v^{(1)} \ge k_v$ and also contradicts to $A_{iu2}=A_{q_v^{(2)}v2},\ A_{pv1}=A_{q_u^{(2)}u1}$ (see (\ref{10})) with any $q_u^{(2)}\ge k_u,\ q_v^{(2)}\ge k_v$.

Note that  (\ref{11}),  (\ref{12}) imply that 
\begin{eqnarray}\label{14}
A_{u1}-A_{v1}\le (q_v-k_v)-(q_u-k_u)\le A_{u2}-A_{v2}.    
\end{eqnarray}
Equalities $A_{iu1}=A_{q_v^{(3)}v1},\ A_{pv2}=A_{q_u^{(3)}u2}$ (see (\ref{8})) for some $q_u^{(3)}\ge k_u,\ q_v^{(3)}\ge k_v$ entail that $A_{iu1}\le A_{q_u^{(3)}u1},\ A_{pv2}\le A_{q_v^{(3)}v2}$ (cf. (\ref{12})), therefore 
$$A_{v1}-A_{u1}\le (q_u^{(3)}-k_u)-(q_v^{(3)}-k_v)\le A_{v2}-A_{u2}$$
\noindent which together with (\ref{14}) deduces equality $A=0$.
This contradicts to genericity of (\ref{4}) and justifies the Lemma in case $i)_{uv}$. \vspace{1mm}

$i)_{uu}$. Let $(S_u=\{i,i_0,q_u\},\ S_v=\{q_v\}),\ 0\le i\le i_0<k_u$ be a minimal solution of (\ref{4}). There are 4 cases (taking into account (\ref{62})):

\begin{align}
    A_{iu1}= A_{q_uu1}&,\ A_{i_0u2}=A_{q_vv2};  \label{15}\\
    A_{iu1}= A_{q_vv1}&,\ A_{i_0u2}=A_{q_uu2};  \label{16}\\
    A_{i_0u1}= A_{q_uu1}&,\ A_{iu2}=A_{q_vv2}; \label{17}\\
    A_{i_0u1}= A_{q_vv1}&,\ A_{iu2}=A_{q_uu2}. \label{18}
\end{align}

Note that in case $i=i_0$ the equations (\ref{15}) coincide with (\ref{17}), and the equations (\ref{16}) coincide with equations (\ref{18}).

We claim that exactly one of (\ref{15})-(\ref{18}) is valid (cf. the proof of the case $i)_{uv}$ above).
W.l.o.g. assume that (\ref{15}) is valid (other three cases can be studied in a similar way). Then due to (\ref{29}) it holds
\begin{eqnarray}\label{19}
A_{iu1}=A_{u1}+q_u-k_u,\ A_{i_0u2}=A_{v2}+q_v-k_v.    
\end{eqnarray}
This provides the expressions for $q_u, q_v$. In addition, we have
\begin{eqnarray}\label{20}
A_{iu1}\le A_{q_vv1},\ A_{i_0u2}\le A_{q_uu2};    
\end{eqnarray}
\begin{eqnarray}\label{21}
A_{iu1}<A_{i_0u1},\ A_{i_0u2}<A_{iu2}.    
\end{eqnarray}
due to (\ref{57}) (in case $i=i_0$ the inequalities (\ref{21}) are omitted).
Observe that (\ref{21}) contradicts to equalities $A_{i_0u1}=A_{q_u^{(1)}u1},\ A_{iu2}=A_{q_v^{(1)}v2}$ (see (\ref{17})) with $q_u^{(1)}\ge k_u,\ q_v^{(1)}\ge k_v$ and also contradicts to equalities $A_{i_0u1}=A_{q_v^{(2)}v1},\ A_{iu2}=A_{q_u^{(2)}u2}$ (see (\ref{18})) with $q_u^{(2)}\ge k_u,\ q_v^{(2)}\ge k_v$.

Note that (\ref{19}), (\ref{20}) imply that 
\begin{eqnarray}\label{22}
A_{u1}-A_{v1}\le (q_v-k_v)-(q_u-k_u)\le A_{u2}-A_{v2}.    
\end{eqnarray}
The equalities $A_{iu1}=A_{q_v^{(3)}v1},\ A_{i_0u2}=A_{q_u^{(3)}u2}$ (see (\ref{16})) with $q_u^{(3)}\ge k_u,\ q_v^{(3)}\ge k_v$ entail that $A_{iu1}\le A_{q_u^{(3)}u1},\ A_{i_0u2}\le A_{q_v^{(3)}v2}$. Whence together with (\ref{19}) we obtain that
$$A_{v1}-A_{u1}\le (q_u^{(3)}-k_u)-(q_v^{(3)}-k_v)\le A_{v2}-A_{u2}$$
\noindent which leads to the equality $A=0$
taking into account (\ref{22}). This contradicts to genericity of (\ref{4}) and justifies the Lemma in case $i)_{uu}$. \vspace{1mm}

The case $i)_{vv}$ is studied similarly to $i)_{uu}$. \vspace{1mm}

$i)_{u}$. Let $(S_u=\{i,q_u\},\ S_v=\{q_v\}),\ 0\le i<k_u$ be a minimal solution of (\ref{4}). There are 4 cases:

\begin{align}
    A_{iu1}=A_{q_vv1}&,\ A_{q_uu2}=A_{q_vv2}; \label{23}\\
    A_{iu1}=A_{q_uu1}&,\ A_{q_uu2}=A_{q_vv2};  \label{24}\\
    A_{q_uu1}=A_{q_vv1}&,\ A_{iu2}=A_{q_vv2}; \label{25}\\
    A_{q_uu1}=A_{q_vv1}&,\ A_{iu2}=A_{q_uu2}. \label{26}
\end{align}

We claim that exactly one of (\ref{23})-(\ref{26}) is valid (cf. the cases $i)_{uv}, i)_{uu}$ above). W.l.o.g. assume that (\ref{23}) is valid (other three cases can be studied in a similar way). Then it holds
\begin{eqnarray}\label{27}
A_{iu1}=A_{v1}+q_v-k_v,\ A_{v2}+q_u-k_u=A_{v2}+q_v-k_v.    
\end{eqnarray}
This provides the expressions for $q_u, q_v$. In addition, we have 
\begin{eqnarray}\label{28}
A_{iu1}\le A_{q_uu1},\ A_{q_uu2}\le A_{iu2}.    
\end{eqnarray}
Together with (\ref{29}), (\ref{27}) this implies that
\begin{eqnarray}\label{30}
A\ge 0,\ A_{iu2}-A_{iu1}\ge A_{v2}-A_{v1}    
\end{eqnarray}

Observe that the equalities $A_{iu1}=A_{q_u^{(1)}u1},\ A_{q_u^{(1)}u2}=A_{q_v^{(1)}v2}$ (see (\ref{24})) with $q_u^{(1)}\ge k_u, q_v^{(1)}\ge k_v$ entail the inequality $A\le 0$
which contradicts to the first inequality of (\ref{30}) because of genericity of (\ref{4}). Similarly, the equalities $A_{q_u^{(2)}u1}=A_{q_v^{(2)}v1},\ A_{iu2}=A_{q_v^{(2)}}$ (see (\ref{25})) with $q_u^{(2)}\ge k_u, q_v^{(2)}\ge k_v$ entail the inequality $A\le 0$
as well, which again leads to a contradiction. 

Finally, suppose that the equalities $A_{q_u^{(3)}u1}=A_{q_v^{(3)}v1},\ A_{iu2}=A_{q_u^{(3)}u2}$ (see (\ref{26})) hold with $q_u^{(3)}\ge k_u, q_v^{(3)}\ge k_v$. Hence due to (\ref{29}) we obtain that $A_{iu2}=A_{u2}+q_u^{(3)}-k_u$. Then the inequality $A_{iu1}\ge A_{q_u^{(3)}u1}$ and (\ref{29}) imply that $A_{iu1}\ge A_{u1}+q_u^{(3)}-k_u=A_{u1}+A_{iu2}-A_{u2}$ which contradicts to (\ref{30}). This justifies the case $i)_{u}$.

The case $i)_v$ is justified in a similar manner.
\end{proof}

In the next lemma we use the notations from Lemma~\ref{unique}.

\begin{lemma}\label{rigidity}
For $0\le i<k_u$ denote $A_i':=A_{iu2}-A_{iu1}-A_{v2}+A_{v1}$. \vspace{1mm}

1. Let $(S_u=\{i,q_u\},\ S_v:=\{q_v\})$ be a minimal solution  of the type $i)_u$ of a generic regular system (\ref{4}). Then the  following inequalities hold:
\begin{align}
    A\ge 0, A_i'\ge 0\ &\text{when}\ S_{u,1}=\{i\}, S_{v,1}=S_{v,2}=\{q_v\}, S_{u,2}=\{q_u\};   \label{43}\\
    A\le 0, A_i'+A\ge 0\ &\text{when}\ S_{u,1}=\{i,q_u\}, S_{v,1}= \emptyset, S_{u,2}=\{q_u\}, S_{v,2}=\{q_v\};  \label{44}\\
    A\le 0, A_i'\le 0\ &\text{when}\ S_{u,1}=\{q_u\}, S_{v,1}=S_{v,2}=\{q_v\}, S_{u,2}=\{i\}; \label{45}\\
    A\ge 0, A_i'+A\le 0\ &\text{when}\ S_{u,1}=\{q_u\}, S_{v,1}=\{q_v\}, S_{u,2}=\{i,q_u\},\ S_{v,2}=\emptyset.\label{46}
\end{align}

2. Let $(S_u', S_v')$ of the type either $i)_{uv}$ or $i)_{uu}$ be a minimal solution of (\ref{4}) such that $i\in S_u'$. Then it holds
\begin{eqnarray}\label{47}
A\ge 0, A_i'+A\ge 0\ \text{when}\ S_{u,1}'=\{i\}, |S_{v,1}'|=1          
\end{eqnarray}
 and either $|S_{u,2}'|=1$ in case $i)_{uv}$  or $|S_{u,2}'|=2$ in case $i)_{uu}$;
 \begin{eqnarray}\label{48}
A\le 0, A_i'\ge 0\ \text{when}\ |S_{u,1}'|=2, S_{v,1}'=\emptyset          
\end{eqnarray}
 and either $S_{u,2}'=\emptyset$ in case $i)_{uv}$  or $|S_{u,2}'|=1$ in case $i)_{uu}$;
 \begin{eqnarray}\label{49}
A\le 0, A_i'+A\le 0\ \text{when}\ S_{u,2}'=\{i\}, |S_{v,2}'|=1          
\end{eqnarray}
 and either $|S_{u,1}'|=1$ in case $i)_{uv}$  or $|S_{u,1}'|=2$ in case $i)_{uu}$;
 \begin{eqnarray}\label{50}
A\ge 0, A_i'\le 0\ \text{when}\ |S_{u,2}'|=2, S_{v,2}'=\emptyset          
\end{eqnarray}
 and either $S_{u,1}'=\emptyset$ in case $i)_{uv}$  or $|S_{u,1}'|=1$ in case $i)_{uu}$. \vspace{1mm}

 3. Assume that both $(S_u,S_v)$ as in 1. and $(S_u',S_v')$ in 2. are minimal solutions of (\ref{4}). If $i\in S_{u,1}$ then $S_{u,1}'=S_{u,1},\ S_{v,1}'=S_{v,1}$, and if $i\in S_{u,2}$ then $S_{u,2}'=S_{u,2},\ S_{v,2}'=S_{v,2}$. \vspace{1mm}

 One can swap the roles of $u$ and of $v$ in 1., 2., 3., simultaneously.
\end{lemma}

\begin{proof}
\begin{enumerate}
    \item The statement (\ref{43}) coincides with (\ref{30}) (see (\ref{23})) in the proof of Lemma~\ref{unique}~$i)_u$. Other statements (\ref{44})-(\ref{46}) are justified in a similar way (cf. the proof of Lemma~\ref{unique}~$i)_{u}$).
    \item The statement (\ref{48}) in case of the type $i)_{uv}$ follows from (\ref{11}), (\ref{13}), (\ref{14}) (see (\ref{7}) in the proof of Lemma~\ref{unique}~$i)_{uv}$). The statement (\ref{48}) in case of the type $i)_{uu}$ follows from (\ref{19}), (\ref{21}), (\ref{22}) (see (\ref{15}) in the proof of Lemma~\ref{unique}~$i)_{uu}$). Other statements (\ref{47}), (\ref{49}), (\ref{50}) are justified in a similar way (cf. the proof of Lemma~\ref{unique}~$i)_{uv}, i)_{uu}$).
    \item The first statement of 3. follows from (\ref{43}), (\ref{44}), (\ref{47}), (\ref{48}), the second statement follows from (\ref{45}), (\ref{46}), (\ref{49}), (\ref{50}).\qedhere
\end{enumerate}
\end{proof}
Now we establish an upper bound on the number of minimal solutions of a generic regular system (\ref{4}).

\begin{theorem}\label{upper}
A generic regular system  (\ref{4}) has at most
$$(k_u+k_v)(k_u+k_v+1)/2$$
\noindent minimal solutions.
\end{theorem}
\begin{proof}
    W.l.o.g. one can assume that $A>0$ (the case $A<0$ can be studied in a similar way).

Denote by $B_{u,l},\ l=1,2$ the set of $i\in \{0,\dots,k_u-1\}$ such that (\ref{4}) has a minimal solution $(S_u,S_v)$ of the type $i)_u$ for which $i\in S_{u,l}$. Then $S_{u,l}=\{i\},\ |S_{v,l}|=1$ when $l=1$, and $|S_{u,l}|=2,\ S_{v,l}=\emptyset$ when $l=2$ since $A>0$ (see Lemma~\ref{rigidity}~1.). It holds $B_{u,1}\cap B_{u,2}=\emptyset$ due to Lemma~\ref{unique}~$i)_u$. Similarly, we define the sets $B_{v,l}\subset \{0,\dots,k_v-1\},\ l=1,2$. Denote $k_{u,l}:=|B_{u,l}|,\ k_{v,l}:=|B_{v,l}|, l=1,2,\ k:= k_{u,1}-k_{v,1}-k_{u,2}+k_{v,2}$.

Due to Lemmas~\ref{unique},~\ref{rigidity} the following upper bounds hold on the  number of minimal solutions of the respective types:
$$\bullet i)_{uv}: k_{u,1}k_{v,2}+k_{u,2}k_{v,1}+(k_u-k_{u,1}-k_{u,2})(k_v-k_{v,1}-k_{v,2})+$$
$$(k_{u,1}+k_{u,2})(k_v-k_{v,1}-k_{v,2})+(k_u-k_{u,1}-k_{u,2})(k_{v,1}+k_{v,2});$$
$$\bullet i)_{uu}: k_{u,1}k_{u,2}+(k_u-k_{u,1}-k_{u,2})(k_u-k_{u,1}-k_{u,2}+1)/2+(k_{u,1}+k_{u,2})(k_u-k_{u,1}-k_{u,2});$$
$$\bullet i)_{vv}: k_{v,1}k_{v,2}+(k_v-k_{v,1}-k_{v,2})(k_v-k_{v,1}-k_{v,2}+1)/2+(k_{v,1}+k_{v,2})(k_v-k_{v,1}-k_{v,2});$$
$$\bullet i)_u: k_{u,1}+k_{u,2};$$
$$\bullet i)_v: k_{v,1}+k_{v,2}.$$

Summing up all these upper bounds and taking into account the inequality $k_{u,1}k_{v,2}+k_{u,2}k_{v,1}+k_{u,1}k_{u,2}+k_{v,1}k_{v,2}\le k^2/4$, we obtain the required upper bound
$$\frac{(k_u+k_v)(k_u+k_v+1)}{2}+\frac{2k-k^2}{4}$$
\noindent (taking into account that $(2k-k^2)/4\le 1/4$).
\end{proof}

\subsection{Lower bound on the number of minimal solutions for $m=n=2$}

The next purpose is to construct a generic regular system (\ref{4}) which contains all minimal solutions of the types $i)_{uv},\ i)_{uu},\ i)_{vv}$ from Lemma~\ref{unique}. First we note some simple growing properties of the maps $A_{-}(P):\mathbb{Z}_{\geq0}\xrightarrow[]{}\mathbb{Z}$ from Definition \ref{def:map_A}.

\begin{remark}\label{increase}
If the coefficients of $P$  as in (\ref{1}) fulfill inequalities $a_{i+1}\ge a_i+1,\ 0\le i<k$ then $A_{j}(P)=j+a_0,\ j\in \mathbb{Z}_{\geq0}$; in particular the map $A_{-}(P)$ is increasing. On the contrary, if $a_{i+1}\le a_i+1$ for $\ 0\le i<k$, then 
$$A_{j}(P)=a_{j},\ 0\le j\le k.$$
\noindent Therefore, if $a_{i+1}<a_i$ for $0\le i<k$ then $A_{P}$ is decreasing on the interval $[0,k]$.
\end{remark}

\begin{theorem}\label{construction}
Assume that in a 
generic 
system (\ref{4}) its coefficients fulfill the following inequalities:
\begin{align}
    a_{i+1,u,1}\ge a_{iu1}+1&,\ 0\le i<k_u; \label{32}\\
    a_{i+1,v,1}\le a_{iv1}-1&,\ 0\le i<k_v;  \label{33}\\
    a_{i+1,u,2}\le a_{iu2}-1&,\ 0\le i<k_u;\label{34}\\
    a_{i+1,v,2}\ge a_{iv2}+1&,\ 0\le i<k_v;\label{35}\\
    a_{0v1}< a_{0u1}&,\  a_{0u2}< a_{0v2}.\label{37}
\end{align}

Then the system (\ref{4}) is regular and has the  minimal solutions of the following configurations:
\begin{itemize}
    \item $(S_u\in \{i,\star\},\ S_v\in \{p,\star\})$ of the type $i)_{uv},\ 0\le i<k_u,\ 0\le p<k_v$;
    \item $(S_u\in \{i,i_0,\star\},\ S_v\in \{\star\})$ of the type $i)_{uu},\ 0\le i\le i_0<k_u$;
    \item $(S_u\in \{\star\},\ S_v\in \{p, p_0, \star\})$ of the type $i)_{vv},\ 0\le p\le p_0<k_v$.
\end{itemize}
\end{theorem}
\begin{proof}
    $i)_{uv}$ Let $0\le i<k_u,\ 0\le p<k_v$. Due to (\ref{34}) there exists a unique $q_u\ge k_u$ such that $A_{iu2}=A_{q_uu2}$ taking into account Remark~\ref{increase}. By the same argument due to (\ref{33}) there exists a unique $q_v\ge k_v$ such that $A_{pv1}=A_{q_vv1}$. We claim that $(S_u:=\{i,q_u\},\ S_v:=\{p,q_v\})$ is a solution of the 
system (\ref{4}) described in the Theorem. To this end it suffices to verify the following inequalities:
\begin{eqnarray}\label{36}
A_{pv1}\le A_{iu1},\  A_{pv1}\le A_{q_uu1},\  A_{iu2}\le A_{pv2},\  A_{iu2}\le A_{q_vv2}.  
\end{eqnarray}
Since $A_{pv1}\le A_{0v1}$ (see (\ref{33})) and $A_{0u1}\le A_{iu1}$ (see (\ref{32})), the first inequality in (\ref{36}) follows from (\ref{37}). Since $A_{q_uu1}\ge A_{0u1}$ (see (\ref{32})) we get the second inequality in (\ref{36}). Since $A_{iu2}\le A_{0u2}$ (see (\ref{34})) and $A_{pv2}\ge A_{0v2}$ (see (\ref{35})), we get the third inequality in (\ref{36}) invoking (\ref{37}). Since $A_{q_vv2}\ge A_{0v2}$ (see (\ref{35})) we get the fourth inequality of (\ref{36}). \vspace{1mm}

$i)_{uu}$ Let $0\le i<i_0<k_u$. Due to (\ref{37}) there exists a unique $q_v\ge k_v$ such that $A_{iu1}=A_{q_vv1}$ (see (\ref{32}), (\ref{33})). Due to (\ref{34}) there exists a unique $q_u\ge k_u$ such that $A_{i_0u2}=A_{q_uu2}$. We claim that $(S_u:=\{i,i_0,q_u\},\ S_v:=\{q_v\})$ is a solution of (\ref{4}).

To this end it suffices to verify the following inequalities:
\begin{eqnarray}\label{38}
A_{iu1}\le A_{i_0u1},\  A_{iu1}\le A_{q_uu1},\  A_{i_0u2}\le A_{iu2},\  A_{i_0u2}\le A_{q_vv2}.  
\end{eqnarray}
The first and the second inequalities in (\ref{38}) follow from (\ref{32}). The third inequality in (\ref{38}) follows from (\ref{34}). The fourth inequality in (\ref{38}) follows from (\ref{37}) taking into account (\ref{34}), (\ref{35}). \vspace{1mm}

$i)_{vv}$ Let $0\le p<p_0<k_v$. Due to (\ref{37}) there exists a unique $q_u\ge k_u$ such that $A_{pv2}=A_{q_uu2}$ (see (\ref{34}), (\ref{35})). Due to (\ref{33}) there exists a unique $q_v\ge k_v$ such that $A_{p_0v1}=A_{q_vv1}$. We claim that $(S_u:=\{q_u\},\ S_v:=\{p,p_0,q_v\})$ is a solution of (\ref{4}). 

To this end it suffices to verify the following inequalities:
\begin{eqnarray}\label{39}
A_{p_0v1}\le A_{pv1},\  A_{p_0v1}\le A_{q_uu1},\  A_{pv2}\le A_{p_0v2},\  A_{pv2}\le A_{q_vv2}.  
\end{eqnarray}
The first inequality in (\ref{39}) follows from (\ref{33}). The second inequality in (\ref{39}) follows from (\ref{37}) taking into account (\ref{32}), (\ref{33}). The third and the fourth inequalities in (\ref{39}) follow from (\ref{35}). \vspace{1mm}

Finally, we have to show that the produced solutions of the types $i)_{uv}, i)_{uu}, i)_{vv}$ are minimal. To this end it suffices to verify that the described
system (\ref{4}) has no solutions of types $i)_u, i)_v$ (see Lemma~\ref{unique}). 

Suppose the contrary and let $(S_u=\{i,q_u\},\ S_v=\{q_v\})$ be a solution of (\ref{4}) of the type $i)_u$. Taking into account (\ref{32})-(\ref{35}) we have to consider two cases: either $A_{iu1}=A_{q_vv1},\ A_{q_uu2}=A_{q_vv2}$ or $A_{iu2}=A_{q_uu2},\ A_{q_uu1}=A_{q_vv1}$. In the former case it holds that $A_{iu2}<A_{q_vv2}$ due to (\ref{37}). In the latter case it holds that $A_{iu1}<A_{q_uu1}$ due to (\ref{32}). Thus, in both cases we arrive to a contradiction with that $(S_u,\ S_v)$ is a solution of (\ref{4}). 

In a similar way a supposition that there exists a solution of the form $(S_u=\{q_u\},\ S_v=\{p,q_v\})$ of (\ref{4}) of the type $i)_v$ leads to a contradiction as well.
\end{proof}

\begin{corollary}\label{gap}
The maximal number of minimal solutions of a generic regular system (\ref{4}) equals $\frac{(k_u+k_v)(k_u+k_v+1)}{2}$.
\end{corollary}

\section{Bounds on the number of minimal solutions of a generic system for $m=n>1$}\label{five}

First, we construct a generic regular system (\ref{4}) for arbitrary $m=n$ with {\it many} minimal solutions, extending Theorem~\ref{construction}.

\begin{theorem}\label{lower}
For each $1\le j,l\le n, 0\le i\le k_j$ take the coefficients $a_{ijl}$ of a generic regular system (\ref{4}) satisfying the following inequalities (cf. Theorem~\ref{construction}):

\begin{eqnarray}\label{51}
a_{i+1, l, l}\le a_{ill}-1,\ 0\le i< k_l;    
\end{eqnarray}
\begin{eqnarray}\label{52}
a_{i+1, l+1 (\text{mod}\ n), l} \ge a_{i,  l+1 (\text{mod}\ n), l}+1,\ 0\le i< k_{l+1 (\text{mod}\ n)};   
\end{eqnarray}
\begin{eqnarray}\label{53}
a_{0,l,l}< a_{0, l+1 (\text{mod}\ n),l};    
\end{eqnarray}
\begin{eqnarray}\label{54}
a_{ijl}> a_{k_{l+1(\text{mod}\ n)}, l+1(\text{mod}\ n), l}+k_j,\ j\neq l, j\neq l+1(\text{mod}\ n).    \end{eqnarray}

Then  for each family of integers $1\le p_1<\cdots < p_s\le n$ such that $p_{j+1}\ge p_j+2(\text{mod}\ n),\ 1\le j\le s$, and for any $0\le i_{j1}\le i_{j2}< k_{p_j},\ 1\le j\le s$ and for any $0\le i_r<k_r,\ 1\le r\le n,\ r\neq p_j,\ r\neq p_j-1(\text{mod}\ n),\ 1\le j\le n$ the system (\ref{4}) has a minimal solution $S$ of (\ref{4}) whose configuration satisfies the following: 
\begin{eqnarray}\label{55}
S_{u_{p_j}}\in \{i_{j1}, i_{j2}, \star\},\ S_{u_{p_j-1(\text{mod}\ n)}}\in \{\star\},\ S_{u_r}\in \{i_r,\star\}.   
\end{eqnarray}
\end{theorem}

\begin{remark}\label{lower_bound}
It holds $|S_{u_1}|+\cdots + |S_{u_n}|=2n$, and the number of minimal solutions of the form (\ref{55}) equals
$$\prod_{1\le j\le s} k_{p_j}(k_{p_j}+1)/2 \cdot \prod_{1\le r\le n, r\neq p_j, r\neq p_j-1(\text{mod}\ n), 1\le j\le s} k_r.$$
\end{remark}

\begin{proof}
    First, let $1\le r\le n, r\neq p_j, r\neq p_j-1(\text{mod}\ n), 1\le j\le s$. Due to (\ref{51}) (cf. also Remark~\ref{increase}) there exists (a unique) $q_r\ge k_r$ such that 
$$A_{q_r,\ r,\ r}=A_{i_r,\ r,\ r}.$$
\noindent Then $S_{u_r}=\{i_r,q_r\}=S_{u_r,r}$.

For $p_j, 1\le j\le s$ due to (\ref{51}) there exists (a unique) $q_{p_j}\ge k_{p_j}$ such that
$$A_{q_{p_j},\ p_j,\ p_j}= A_{i_{j2},\ p_j,\ p_j}.$$
\noindent Due to (\ref{51})-(\ref{53}) there exists (a unique) $q_{p_j-1(\text{mod}\ n)}\ge k_{p_j-1(\text{mod}\ n)}$ such that 
$$A_{q_{p_j-1(\text{mod}\ n)},\ p_j-1(\text{mod}\ n),\ p_j-1(\text{mod}\ n)}= A_{i_{j1},\ p_j,\ p_j-1(\text{mod}\ n)}.$$

Then it holds 
$$S_{u_{p_j}}=\{i_{j1},\ i_{j2},\ q_{p_j}\},\ S_{u_{p_j-1(\text{mod}\ n)}}=\{q_{p_j-1(\text{mod}\ n)}\},$$
\noindent thus
$$S_{u_{p_j},\ p_j}= \{i_{j2}, q_{p_j}\},\ S_{u_{p_j},\ p_j-1(\text{mod}\ n)}= \{i_{j1}\},\ S_{u_{p_j-1(\text{mod}\ n)},\ p_j-1(\text{mod}\ n)}= \{q_{p_j-1(\text{mod}\ n)}\}.$$

Taking into account (\ref{51})-(\ref{54}) one can verify (\ref{55}) and that $(S_{u_1},\dots,S_{u_n})$ is a solution of (\ref{4}) similar to the proof of Theorem~\ref{construction}, and once again, similar to the proof of Theorem~\ref{construction} one can verify that the produced solution is minimal.
\end{proof}

Similar to Lemma~\ref{unique} one can prove the following lemma.

\begin{lemma}\label{unique_extended}
Let $(S_{u_1},\dots,S_{u_n})$ be a minimal solution of a generic regular system (\ref{4}). Then for any $1\le j\le n$ there exists a unique $q_j\ge k_j$ such that $q_j\in S_{u_j}$. It holds $|S_{u_1}|+\cdots +|S_{u_n}|\le 2n$. The subsets $S_{u_j,l}\subset S_{u_j},\ 1\le j,l\le n$ are uniquely defined by the sets 
$S_{u_1}\cap \{0,\dots,k_1-1\},\dots,S_{u_n}\cap \{0,\dots,k_n-1\}$.
\end{lemma}

\begin{proof}
    For $1\le l\le n$ assume that the minimum (which is unique due to  regularity of (\ref{4})) among $\{A_{ijl}\ |\ i\in S_{u_j}, 0\le i<k_j, 1\le j\le n\}$ is attained at 
$A_{i_0, j_0,l}$.
One can verify (cf. the proof of Lemma~\ref{unique}) that $i_0\in S_{u_{j_0}, l}$. This determines the intersections $S_{u_j, l}\cap \{0,\dots, k_j-1\},\ 1\le j,l\le n$.

For each $1\le l\le n,\ i\in S_{u_{j_0}, l},\ 0\le i<k_{j_0},\ 1\le j, j_0\le n$ take (a unique) $q_{jl}\ge k_j$ such that $A_{q_{jl}, j, l}=A_{i, j_0, l}$, provided that $q_{jl}$ does exist. For each $1\le j\le n$ denote by $q_j$ the maximum (which is unique due to genericity of (\ref{4}), see (\ref{62})) among existing $q_{jl},\ 1\le l\le n$. One can verify that $q_j\in S_{u_j}$ (cf. the proof of Lemma~\ref{unique}). This determines the sets $S_{u_j},\ 1\le j\le n$.

We claim that the sets $S_{u_j,l},\ 1\le j,l\le n$ are determined in a unique way. Suppose the contrary. Then there exist two different permutations $p_1, p_2\in Sym(n)$ such that 
$$A_{i, j_0, l}=A_{p_1(l), l}+q_{p_1(l)}-k_{p_1(l)}=A_{p_2(l), l}+q_{p_2(l)}-k_{p_2(l)},\ 1\le l\le n$$
\noindent (see (\ref{29})). Summing up over $1\le l\le n$ the second and the third parts of the latter equalities, we arrive to a contradiction with genericity of (\ref{4}), see (\ref{61}).
\end{proof}

Relying on Lemma~\ref{unique_extended} one can bound from above the number of minimal solutions of a generic regular system (\ref{4}) by 
\begin{eqnarray}\label{60}
\sum_{d_1+\cdots+d_n\le n} \binom{k_1+d_1-1}{d_1} \cdots \binom{k_n+d_n-1}{d_n}.
  \end{eqnarray}
\noindent Unlike the case $n=2$ (see Corollary~\ref{gap}) the gap between the obtained lower bound (see Theorem~\ref{lower} and Remark~\ref{lower_bound}) and the upper bound is quite big. It would be interesting to diminish this gap.

\section{Inversions of families of permutations and an upper bound on the number of minimal solutions of a generic system}\label{six}
In this section we improve the upper bound (\ref{60}) on the number of minimal solutions of a generic regular system (\ref{4}). First we extend the concept of an inversion to families of permutations and prove an upper bound on the number of inversions. We view a permutation from $Sym(r)$ as a bijection of the set $[r]$ with itself.

\begin{definition}\label{def:inversion}
For  $w\in Sym(r)$ and an $n$-tuple $1\le i_1, \dots, i_n\le r$, denote 
$$m_w(i_1,\dots, i_n):= w^{-1}(\min \{w(i_1), \dots, w(i_n)\}).$$

\noindent We say that $\{i_1, \dots, i_n\}$ is an {\bf inversion} of a family of permutations $w_1, \dots, w_n \in Sym(r)$ if integers $m_{w_1}(i_1,\dots,i_n),\dots, m_{w_n}(i_1,\dots,i_n)$ are pairwise distinct. Clearly, in this case it holds $\{m_{w_1}(i_1,\dots,i_n),\dots, m_{w_n}(i_1,\dots,i_n)\}=\{i_1,\dots,i_n\}$.
\end{definition}

The concept of inversions of a family of permutations generalizes the definition of an inversion of a permutation $w$ which coincides with inversions of a pair of permutations $e,w$ where $e$ denotes the identical permutation.

\begin{theorem}\label{inversion}
Consider $n\ge 2$.  The number of inversions of a family $w_1,\dots,w_n \in Sym(r)$ is bounded by 
$$\frac{2r^n}{n\cdot n!} + O(r^{n-1})$$
\noindent for a fixed $n$ and growing $r$.
\end{theorem}

\begin{remark}
 i) For $n=2$ it is well known that the maximal number of inversions equals  $r(r-1)/2$.

 ii) An obvious bound on the number of inversions is
 $$\binom{r}{n} \le \frac{r^n}{n!}+O(r^{n-1}).$$
 Thereby, Theorem~\ref{inversion} improves the obvious bound asymptotically in $n/2$ times.
\end{remark}

\begin{lemma}\label{cut}
Consider $n\ge 3$. For a family  $w_1,\dots,w_n \in Sym(l)$ there exists $1\le j\le n$ such that the family $\{w_1,\dots,w_n\}\setminus \{w_j\}$ has at most of 
$$\frac{2l^{n-1}}{n!}+O(l^{n-2})$$
\noindent inversions.
\end{lemma}

{\bf Proof of the lemma}. If $1\le i_1,\dots, i_{n-1}\le l$ is an inversion of a family $\{w_1,\dots,w_n\}\setminus \{w_j\}$ then among $n$ integers $m_{w_1}(i_1,\dots,i_{n-1}),\dots, m_{w_n}(i_1,\dots,i_{n-1})$ there are $n-1$ pairwise distinct coinciding with $i_1,\dots,i_{n-1}$. Therefore there exist exactly two integers $1\le j\neq j_0\le n$ such that $\{i_1,\dots, i_{n-1}\}$ is an inversion of the family $\{w_1 ,\dots,w_n\}\setminus \{w_j\}$ and of the family $\{w_1,\dots,w_n\}\setminus \{w_{j_0}\}$ as well. In other words, $m_{w_j}(i_1,\dots, i_{n-1})=m_{w_{j_0}}(i_1,\dots, i_{n-1})$.

This completes the proof of the lemma taking into account that there are 
$$\frac{l^{n-1}}{(n-1)!}+O(l^{n-2})$$
\noindent $(n-1)$-tuples of the form $1\le i_1,\dots, i_{n-1}\le l$ and that there are exactly $n$ subsets of the size $n-1$ of the set $\{w_1,\dots,w_n\}$. $\Box$

{\bf Proof of Theorem~\ref{inversion}}. We proceed by induction on $r$. W.l.o.g. one can assume that $n\ge 3$. Denote $i_j:=w_j^{-1}(1),\ 1\le j\le n$. If $i_1,\dots, i_n$ are not pairwise distinct, say $i_1=i_2$, then there are no inversions which contain $i_1$. Thus, w.l.o.g. one can assume that  $i_1,\dots, i_n$ are pairwise distinct.

There is the unique bijection

$$R:=R_{i_1,\dots,i_n}:\{1,\dots, r-n\}  \xrightarrow[]{\sim} \{1,\dots,r\} \setminus \{i_1,\dots,i_n\}$$
\noindent which preserves the order. For any permutation $w\in Sym(r)$ denote 
$$\overline{w}:= R^{-1}\circ w\circ R_{i_1,\dots,i_n}\in Sym(r-n)$$
\noindent (note that $R^{-1}$ is defined for any triple $\{i_1', \dots, i_n'\}$ in place of $\{i_1, \dots, i_n\}$). Apply Lemma~\ref{cut}  to the family $\{\overline{w_1}, \dots, \overline{w_n}\}$ and suppose for definiteness that the family $\{\overline{w_1}, \dots, \overline{w_{n-1}}\}$ has at most of
\begin{eqnarray}\label{58}
 \frac{2(r-n)^{n-1}}{n!}+O((r-n)^{n-2})   
\end{eqnarray}
\noindent inversions.

Denote by $I(r)$ the maximal number of inversions in $n$-families from $Sym(r)$. The number of inversions $\{i_1',\dots, i_n'\}$ of the family $w_1, \dots, w_n$ such that

\begin{itemize}
    \item $i_n\notin \{i_1',\dots, i_n'\}$ does not exceed $I(r-1)$;
    \item $i_n \in \{i_1',\dots,\ i_n'\},\  \{i_1,\dots, i_{n-1}\} \cap \{i_1',\dots,\ i_n'\}= \emptyset$ is bounded by (\ref{58});
    \item $i_n \in \{i_1',\dots,\ i_n'\},\ \{i_1,\dots, i_{n-1}\} \cap \{i_1',\dots,\ i_n'\} \neq \emptyset$ is less than $O(r^{n-2})$.
\end{itemize}

Thus, we get an inductive inequality 
$$I(r)\le I(r-1)+\frac{2(r-n)^{n-1}}{n!}+O(r^{n-2}),$$
which completes the proof of the theorem.$\Box$


\begin{theorem}\label{upper_improved}
A generic regular system (\ref{4}) has at most of
$$\frac{2(k_1+\cdots +k_n)^n}{n\cdot n!}+O((k_1+\cdots +k_n)^{n-1})$$
\noindent minimal solutions for a fixed $n$ and growing $k_1+\cdots +k_n$. 
\end{theorem}

\begin{proof}
    We define the following permutations $w_1, \dots, w_n \in Sym(k_1+\cdots+k_n)$. The permutation $w_l, 1\le l\le n$ corresponds to the sequence (\ref{57})
in the increasing order. In other words, $w_l(j)$ equals the place in the increasing order of $j$-th element of the  sequence (\ref{57}). Note that the elements of the sequence (\ref{57}) are pairwise distinct because of regularity of the system (\ref{4}).

Observe that for any minimal solution of a configuration $(S_{u_1}, \dots, S_{u_n})$ such that 
\begin{equation}\label{59}
 \star \in S_{u_1} \cap \cdots \cap S_{u_n}, \sum_{j=1}^n|S_{u_j}|=2n, \{i_1, \dots, i_n\}= S_{u_1} \cup \cdots \cup S_{u_n}  \setminus \{\star\}   
\end{equation}
it holds that $\{i_1, \dots,i_n\}$ is an inversion of the family of permutations $w_1, \dots, w_n$  (cf. (\ref{13}), (\ref{21})).

Thus, the number of minimal solutions of a configuration (\ref{59}) is bounded by
$$\frac{2(k_1+\cdots+k_n)^n}{n\cdot n!} + O((k_1+\cdots+k_n)^{n-1})$$ 
due to Theorem~\ref{inversion} taking into account Lemma~\ref{unique_extended}. The number of minimal solutions of a configuration $(S_{u_1}, \dots, S_{u_n})$ such that 
$$\star \in S_{u_1} \cap \dots \cap S_{u_n},\ |S_{u_1}|+\cdots +|S_{u_n}|\le 2n-1$$
\noindent is less than $O((k_1+\cdots +k_n)^{n-1})$. This completes the proof of the theorem.
\end{proof}

\begin{remark}
i) The bound from Theorem~\ref{upper_improved} improves the upper bound (\ref{60}) asymptotically in $n/2$ times.

ii) In case $n=3$ and $k:=k_1=k_2=k_3$ we obtain a lower bound $5k^3/2-O(k^2)$ (Theorem~\ref{lower}) on the number of minimal solutions of a generic regular system (\ref{4}), and the upper bound $3k^3+ O(k^2)$ (Theorem~\ref{upper_improved}).    
\end{remark}

\section*{Acknowledgements}
This work was realized while C. G. was visiting the Vietnam Institute for Advanced Studies in Mathematics (VIASM) during the last three months of 2024, and he wishes to thank the Institution for its hospitality and financial support.

\par\bigskip\noindent
\textsc{CNRS,  Math\'ematique, Universit\'e de Lille, Villeneuve
d'Ascq, 59655, France}
\par\noindent
e-mail address: \verb+dmitry.grigoryev@univ-lille.fr+
\vspace{.3cm}
\par\noindent
\textsc{Centro de Investigaci\'on en Matem\'aticas, A.C., Jalisco S/N, Col. Valenciana CP: 36023 Guanajuato, Gto, M\'exico.
}
\par\noindent
e-mail address: \verb+cristhian.garay@cimat.mx+

\end{document}